%% file: nonisothermal.tex
\documentclass[smallextended]{svjour3}       
\usepackage{amsmath}
\usepackage{amssymb}

\usepackage[english]{babel}
\usepackage[utf8]{inputenc}

\smartqed  

\usepackage{graphicx}
\usepackage{xcolor}

\definecolor{LightGray}{HTML}{c7c7c7}
\definecolor{Blue}{HTML}{42bbed}

\usepackage{mathbbol}
\usepackage{bm} 

\usepackage{units}
\usepackage{tensor}

\usepackage{lineno}

\usepackage[color=LightGray]{todonotes}

\usepackage{multicol}

\usepackage{comment}

\usepackage{placeins}

\input{macros}

\newcommand{\AAA}[1]{{\color{red}#1}}
\newcommand{\BBB}[1]{{\color{gray}#1}}

\newcommand{\detY}{Y} 

\numberwithin{equation}{section}

\includecomment{withsummaries}

\includecomment{withreverse}

\includecomment{withtheorem}

\begin{document}

\title{
  Non-isothermal viscoelastic 
  flows with 
  conservation laws and relaxation 
}

\author{S\'ebastien Boyaval \and Mark Dostal\'{\i}k}

\institute{
S\'ebastien Boyaval 
\at
Laboratoire d'hydraulique Saint-Venant (Ecole des Ponts -- EDF R\& D -- CEREMA)\\ \& Matherials (Inria Paris)\\
EDF'lab Chatou - 6 quai Wattier\\
Chatou Cedex 78401, France\\
\email{sebastien.boyaval@enpc.fr} 
\and
Mark Dostal\'{\i}k 
\at 
Mathematical Institute\\
Faculty of Mathematics and Physics\\
Charles University\\
Sokolovsk\'a 83\\
Prague 186\;75, Czechia\\
\email{dostalik@karlin.mff.cuni.cz}
}

\date{\today} 
 
\maketitle 

\begin{abstract}
    \input{abstract}

\end{abstract}




\input{text}

\begin{acknowledgements}
S\'ebastien Boyaval has been supported by ANR JCJC SEDIFLO Project-ANR-15-CE01-0013.

Mark Dostal\'ik has been supported by the Czech Science Foundation, Grant Number 20-11027X. Additional funding was provided by institutional grants Charles University Grant Agency, Grant Number 1652119, and by Charles University Research Programme No. UNCE/SCI/023.
\end{acknowledgements}


\bibliographystyle{spmpsci}      
\bibliography{bibliography}

\newpage
\appendix


\end{document}

%% file: macros.tex
\let\originalleft\left
\let\originalright\right
\renewcommand{\left}{\mathopen{}\mathclose\bgroup\originalleft}
\renewcommand{\right}{\aftergroup\egroup\originalright}

\newcommand*{\texteq}[1]{\stackrel{\text{#1}}{=}}
\newcommand*{\defeq}{\texteq{def}}

\newcommand{\kboltzmann}{k_{\mathrm{B}}}

\newcommand{\dragcoef}{\zeta}

\DeclareMathOperator{\divergence}{div}

\DeclareMathOperator{\rot}{rot}

\DeclareMathOperator{\tr}{tr}

\DeclareMathOperator{\Cof}{Cof} 

\newcommand{\exponential}[1]{\ensuremath{{\mathrm e}^{#1}}}

\newcommand{\reference}{\mathrm{ref}}

\newcommand{\traceless}[1]{{#1}_{\delta}}

\newcommand{\diff}{\mathrm{d}}

\newcommand{\myvec}[1]{\ensuremath{\mathbf{#1}}}
\newcommand{\greekvec}[1]{\ensuremath{\boldsymbol{#1}}}
\newcommand{\tensorq}[1]{\ensuremath{\mathbf{#1}}}      
\newcommand{\tensorc}[1]{\ensuremath{\mathrm{#1}}}      

\newcommand{\transpose}[1]{#1^\top}
\newcommand{\transposei}[1]{#1^{-\top}}
\newcommand{\inverse}[1]{#1^{-1}}

\newcommand{\identity}{\ensuremath{\tensorq{I}}}

\newcommand{\cstress}{\tensorq{T}}

\newcommand{\ctensor}{\tensorq{C}}

\newcommand{\deformation}{\greekvec{\chi}}

\newcommand{\fgrad}{\tensorq{F}}
\newcommand{\fgradc}{\tensorc{F}}

\makeatletter
\@ifpackageloaded{bm}%
{%
}{%

}

\@ifpackageloaded{bm}%
{%
 
}{%

}

\@ifpackageloaded{bm}%
{%
}{%

}
\makeatother

\newcommand{\vecv}{\ensuremath{\myvec{v}}}

\newcommand{\gradsym}{\ensuremath{\tensorq{D}}}

\newcommand{\gradvl}{\ensuremath{\tensorq{L}}}

\newcommand{\vecvc}{\tensorc{v}}

\newcommand{\R}{\ensuremath{{\mathbb R}}}

\newcommand{\N}{\ensuremath{{\mathbb N}}}

\newcommand{\ienergy}{\ensuremath{e}} 
\newcommand{\ienergyref}{\ensuremath{e_{\reference}}}
\newcommand{\fenergy}{\ensuremath{\psi}} 
\newcommand{\entropy}{\ensuremath{\eta}} 

\newcommand{\temp}{\ensuremath{\theta}} 
\newcommand{\tempref}{\ensuremath{\theta_{\reference}}}
\newcommand{\pressure}{p}

\newcommand{\mns}{\ensuremath{m}} 

\newcommand{\cheatvolsol}{\ensuremath{c_{\mathrm{V}, \mathrm{s}}}}
\newcommand{\cheatvolref}{\ensuremath{c_{\mathrm{V}, \reference}}} 

\newcommand{\density}{\rho}
\newcommand{\densityref}{\rho_{\reference}}
\newcommand{\rhor}{\rho_{\mathrm{R}}}

\newcommand{\efluxc}{\myvec{j}_{e}} 

\newcommand{\entfluxc}{\myvec{j}_{\entropy}} 

\newcommand{\entprodc}{\xi} 

\newcommand{\pd}[2]{\ensuremath{\frac{\partial {#1}}{\partial {#2}}}}
\newcommand{\ppd}[2]{\ensuremath{\frac{\partial^2 {#1}}{\partial {#2^2}}}}
\newcommand{\dd}[2]{\ensuremath{\frac{\diff {#1}}{\diff {#2}}}}

\newcommand{\fid}[1]{\stackrel{\nabla}{#1}} 

\newcommand{\lfid}[1]{\stackrel{\Delta}{#1}}

\makeatletter
\@ifpackageloaded{tensor}
{

}{%

}
\makeatother

\makeatletter
\@ifpackageloaded{tensor}
{

}{%

}
\makeatother

\newcommand{\absnorm}[1]{\ensuremath{\left|#1\right|}}

\makeatletter
\@ifundefined{volume}{%
}%
{%
}
\makeatother

\makeatletter
\@ifpackageloaded{MnSymbol} 
{
\newcommand{\tensordot}[2]{\ensuremath{#1 \vdotdot #2}} 
}{%
\newcommand{\tensordot}[2]{\ensuremath{#1 : #2}} 
}
\makeatother

\newcommand{\vectordot}[2]{\ensuremath{#1 \cdot #2}}

\newcommand{\tensorf}[1]{{\mathfrak{#1}}}


%% file: abstract.tex
We propose a system of 
conservation laws with 
relaxation source terms (i.e. 
balance laws) for 
non-isothermal 
viscoelastic flows of 
Maxwell fluids. 

The system is an extension of the polyconvex elastodynamics of hyperelastic bodies using additional structure variables.
It is obtained by writing the Helmholtz free energy as the sum of a volumetric energy density 
(function of the determinant of the deformation gradient $\det\fgrad$ and the 
temperature $\temp$ like the standard perfect-gas law or Noble--Abel stiffened-gas law) 
plus a polyconvex 
strain energy density 
function of $\fgrad$, $\temp$
and of symmetric positive-definite structure tensors that relax at a characteristic time scale.
%
One 
feature of our model is that it unifies various 
ideal materials ranging from hyperelastic solids to 
perfect fluids, 
encompassing fluids 
with memory like Maxwell fluids.

We establish a strictly convex mathematical entropy 
to show that the 
system 
is 
symmetric-hyperbolic. 
Another feature of the proposed 
model is therefore the short-time 
existence and uniqueness of smooth solutions, 
which define genuinely causal viscoelastic flows with 
waves 
propagating at finite speed.

In heat-conductors, we complement the system by a Maxwell--Cattaneo equation for an energy-flux variable.
The system is still symmetric-hyperbolic, and 
smooth evolutions 
with finite-speed waves remain well-defined.

%% file: text.tex
\section{Introduction}
\label{sec:introduction}

Continuum mechanics has proved a useful theory 
for a number of real materials, when
constitutive assumptions
allow one to predict motions 
unequivocally 
see e.g. \cite{marsden.je.hughes.tjr:mathematical,boyaval:viscoelastic}. 
However, for many materials, it remains a challenge to propose constitutive assumptions that both yield unequivocal predictions and realistic behaviours. 
It is the case of polymeric materials which exhibit viscoelastic stress relaxation 
see e.g. \cite{mackay.at.phillips.tn:on,bollada.pc.phillips.tn:on}.

Recently, one of us proposed a symmetric-hyperbolic system of conservation laws with relaxation source terms 
for the isentropic or isothermal compressible viscoelastic flows of Maxwell fluid \cite{boyaval:viscoelastic}. 
The system basically extends the polyconvex elastodynamics conservation laws of hyperelastic materials 
by introducing new structure variables 
and using Maxwell definition of viscosity. 
It 
contains a compressible version of the 
Upper-Convected Maxwell model, which is 
known to describe 
polymer flows
see e.g. \cite{1.549663,mackay.at.phillips.tn:on}. 
A salient feature of the proposed model is the existence and uniqueness (on short-times) of smooth solutions to the initial-value problem,
with waves 
propagating at finite-speed. 
However, for practical applications, many polymer flows are non-isothermal. 

\smallskip

Here, we first propose an extension of the recent model of \cite{boyaval:viscoelastic} 
to \emph{non-isothermal} flows. 
This is achieved by letting the two terms in the Helmholtz free energy 
of \cite{boyaval:viscoelastic} depend on the temperature $\temp$.
On the one hand, the volumetric 
term function of the determinant of the deformation gradient $\det\fgrad$ 
can be chosen as the standard perfect-gas law (or Noble-Abel stiffened-gas law, see variations of our model in Section~\ref{sec:generalizations-to-other-models}).
On the other hand, the polyconvex 
Hookean energy density term that is a function of $\fgrad$ 
and of positive-definite structure tensors 
has a molecular interpretation (see e.g. \cite{dressler.m.edwards.bj.ea:macroscopic,boyaval:viscoelastic})
which suggests how it depends on $\temp$.
The dependence on $\temp$ 
is also suggested by 
phenomenological models for rubbery 
materials  \cite{CHADWICK1984}. 
%
%
Heat-conduction is allowed at finite propagation-speed using Maxwell-Cattaneo law.

We propose a 
\emph{symmetric-hyperbolic} system of conservation laws with relaxation. 
Therefore, the system unequivocally models smooth non-isothermal 
viscoelastic flows (on short times) using finite-speed waves.

\smallskip

Second, 
we consider important variations of our model.
For instance, we add a term function of the \emph{inverse} deformation gradient 
to the polyconvex 
Hookean elastic energy density, 
as in models 
for 
polymer melts 
like the celebrated K-BKZ fluids. 
It allows for an equivalent Lagrangian description of our Eulerian model,
and in turn the 
consistency of our smooth solutions 
to the 
governing system of conservation laws 
with relaxation \cite{wagner:symm}. 
We also 
describe how to further modify the polyconvex 
strain energy term 
and account for finite-extensibility 
e.g. in applications to polymer flows,
using what is called Gent law for rubbery materials or FENE-P law for diluted polymer suspensions.

\smallskip

The paper is organized as follows. 
In Section \ref{sec:compressible-heat-conducting-maxwell-model} we first recall the 
continuum thermomechanics setting and the Maxwell--Cattaneo heat-conduction model.
Then, we proceed with the derivation of 
constitutive relations in a thermodynamically consistent way, 
looking for closure relations that guarantee the nonnegativity of the entropy production. 
(We presume 
that the identification of the \emph{energy storage mechanisms} and \emph{entropy storage mechanisms} provides a complete characterisation of the fluid,
deferring to the Section \ref{sec:short-time-well-posedness-for-compressible-heat-conducting-maxwell-fluid} a mathematical proof
that thermomechanical evolutions 
are unequivocally well-defined.) 

In Section \ref{sec:short-time-well-posedness-for-compressible-heat-conducting-maxwell-fluid}
auxiliary tensorial quantities are introduced, which can be interpreted as structure variables (a material metric)
and which enable us to find an additional balance law for a strictly convex scalar quantity.
As an application of the theory of first-order systems of conservation laws with 
source terms,
that additional law yields symmetrizability of our system of balance laws, and short-time well-posedness of the associated Cauchy problem.

In Section \ref{sec:generalizations-to-other-models} we discuss generalizations of our compressible heat-conducting model for Maxwell fluids
to (i) models with an equivalent Lagrangian description, in the spirit of other fluids of the K-BKZ class for polymer melts, 
to (ii) models with finite-extensibility effects, like in FENE-P fluids or Gent elastomers, 
and to (iii) models with a volumetric energy term more suitable for liquids than the perfect-gas law, using Noble--Abel stiffened-gas law namely.

\section{Compressible Maxwell model, with heat-conduction}
\label{sec:compressible-heat-conducting-maxwell-model}

\subsection{Balance laws of continuum thermomechanics}
\label{sec:balance-laws-of-continuum-thermomechanics}
Let us briefly recall the 
concepts of thermomechanics of continuous media while introducing the notation. We refer to e.g. \cite{marsden.je.hughes.tjr:mathematical,silhavy.m:mechanics} for a proper introduction to the theory. To begin with, we assume that the balance of total energy holds,
\begin{equation}
    \label{eq:balance-of-total-energy}
    \density \dd{E}{t} 
    =
    \divergence \left( \transpose{\cstress} \vecv - \efluxc \right)
    +
    \density \vectordot{\myvec{f}}{\vecv},
\end{equation}
where $\density$ denotes the density of the fluid, $E$ denotes its specific total energy, $\cstress$ denotes the Cauchy stress tensor, $\vecv$ denotes the spatial velocity field, $\efluxc$ denotes the energy flux, and $\myvec{f}$ denotes the specific body force. Note that we assume an Eulerian description in Euclidean space
with divergence operator
$\divergence$ and the material derivative
\begin{equation}
  \label{eq:material-derivative}
  \dd{}{t}
  \defeq
  \pd{}{t}
  +
  \vectordot{\vecv}{\nabla}.
\end{equation}
Assuming that the specific total energy $E$ is given by 
\begin{equation}
  \label{eq:total-energy}
  E
  \defeq
  e
  +
  \frac{1}{2} \absnorm{\vecv}^2,
\end{equation}
where $e$ denotes the specific internal energy of the material, 
and that the balance of total energy \eqref{eq:balance-of-total-energy} is Galilean invariant,
one can 
derive 
\begin{subequations}
  \begin{align}
      \label{eq:balance-of-mass}
      \dd{\density}{t} 
      &= - \density \divergence \vecv,
      \\
      \label{eq:balance-of-linear-momentum}
      \density \dd{\vecv}{t} 
      &=
      \divergence \cstress + \density \myvec{f},
  \end{align}
\end{subequations}
i.e. the balance of mass and the balance of linear momentum, respectively. 
An 
equation for the specific internal energy can be deduced from 
\eqref{eq:balance-of-linear-momentum} and \eqref{eq:balance-of-total-energy}
\begin{equation}
  \label{eq:ienergy-evolution-equation}
  \density \dd{\ienergy}{t}
  +
  \divergence \efluxc
  =
  \tensordot{\cstress}{\gradvl},
\end{equation}
where $\gradvl$ denotes the velocity gradient, i.e. $\gradvl \defeq \nabla \vecv$.

Here, $\tensordot{\tensorq{A}}{\tensorq{B}} \defeq \tr(\transpose{\tensorq{A}} \tensorq{B})$ denotes the Frobenius inner product as usual.

Lastly, we assume the evolution equation for the specific entropy $\entropy$,
\begin{equation}
  \label{eq:entropy-evolution-equation}
  \density \dd{\entropy}{t}
  +
  \divergence \entfluxc
  =
  \entprodc,
\end{equation}
where $\entfluxc$ denotes the entropy flux, and $\entprodc$ denotes the entropy production of the given fluid. 
Note that to guarantee the second law of thermodynamics, the entropy production must be a nonnegative quantity.
We also require a well-defined, 
positive thermodynamic temperature $\temp\defeq\partial_{\entropy}\ienergy>0$.

Our goal here is to 
predict ``viscoelastic'' fluid motions 
as the solutions to Cauchy initial-value problems that use the equations above, plus closure relations between the physical quantities introduced. 
By viscoelastic motions, we 
mean motions where the stress tensor is governed by an evolution 
with features similar to the seminal 
viscoelastic model of Maxwell \cite{Maxwell01011867}, see also \cite{boyaval:viscoelastic}.

\subsection{Thermodynamic 
of compressible Maxwell fluids, with heat conduction}
\label{sec:thermodynamic-background}

\subsubsection{Specific Helmholtz free energy}
\label{sec:specific-helmholtz-free-energy}

For 
\emph{compressible Maxwell fluids} 
we propose 
a \emph{fundamental thermodynamic relation} (or \emph{complete equation of state}) given by a specific Helmholtz free energy $\fenergy$
that is the sum of three terms as follows
\begin{equation}
  \label{eq:fenergy}
  \fenergy(\density, \temp, \ctensor, \efluxc) 
  \defeq
  \fenergy_{\mathrm{s}}(\density, \temp)
  +
  \frac{\alpha}{2} \left( K(\temp) \tr \ctensor - \kboltzmann \temp \log \det \ctensor \right)
  +
  \frac{\tau_0}{2 \kappa} \absnorm{\efluxc}^2 \,.
\end{equation}

We have denoted $\kboltzmann$ the Boltzmann constant, 
$\alpha > 0$ a degree of elasticity per unit mass, $\tau_0 > 0$ a relaxation time and $\kappa > 0$ a thermal conductivity. 

For 
\emph{heat-conducting} fluids, we have added a dependence of $\fenergy$
on internal degrees of freedom through an energy flux vector $\efluxc$. 
The particular case without heat conduction coincides with the formal limit $\efluxc=\kappa \nabla \temp$, $\kappa\to0$.

In any case, 
the fundamental variables of the specific Helmholtz free energy other than the energy flux vector $\efluxc$ 
are
the temperature $\temp$, and the density $\density$ together with an additional symmetric positive-definite tensorial quantity $\ctensor$ 
to be linked further with the 
deformations 
of the 
body (a measure of strain 
beyond the purely volumetric term $\rho$).

The volumetric 
contribution $\fenergy_{\mathrm{s}}$ 
to the specific Helmholtz free energy $\fenergy$ is for the moment left unspecified. 
However, the choice of $\fenergy_{\mathrm{s}}$ 
matters in well-posedness. 
We use particular (analytical) 
formulae 
in Sections \ref{sec:governing-equations-symmetrizable-system-balance-laws} and \ref{sec:generalizations-to-other-models}.

The second term in the right-hand side of \eqref{eq:fenergy} 
characterizes Hookean-elastic fluids with stiffness $K(\theta)>0$ like Maxwell fluids. For polymer suspensions, it can be derived from a 
molecular theory when $\ctensor$ 
is 
the conformation tensor of diluted Hookean-elastic dumbbells 
with 
spring factor $K(\theta)$, see e.g. \cite{boyaval:viscoelastic} for a reference.
The tensor then expresses the degree of orientation and elongation of diluted polymer chains,
and following the seminal work on non-isothermal viscoelastic fluid models \cite{marrucci.g:free} we can assume here
\begin{equation}
  \label{eq:elastic-spring-factor-affine-function}
  K(\temp)
  \defeq
  K_0
  +
  K_1 \temp,
\end{equation}
with constant parameters $K_0, K_1 > 0$. 
Moreover, $\ctensor$ 
has the meaning of a strain measure similar to the left Cauchy-Green deformation tensor $\fgrad\transpose{\fgrad}$,
$$
\fgrad \defeq \pd{\deformation_t}{\myvec{a}}\circ\deformation_t^{-1}
$$ 
being the gradient with respect to material coordinates $\myvec{a}$ 
of a deformation $\deformation_t: \myvec{a}\in \R^3 \to \R^3\ni \myvec{x}$ 
associated with the continuous motion of velocity $\vecv$:
\begin{equation}
  \label{eq:flf}
  \dd{\fgrad}{t}
  =
  \gradvl \fgrad.
\end{equation}
A natural time rate for $\ctensor$ is then given by the \emph{upper convected 
derivative}
\begin{equation}
  \label{eq:upper-convected-derivative}
  \fid{\ctensor}
  \defeq
  \dd{\ctensor}{t}
  -
  \gradvl \ctensor
  -
  \ctensor \transpose{\gradvl},
\end{equation}
an objective time derivative which ensures material frame-indifference of our constitutive assumptions.

Lastly, for \emph{entangled} polymer melts and other rubber-like materials, 
one rather interprets $\ctensor$
as a measure of strain similar to the inverse right Cauchy-Green 
tensor $\transposei{\fgrad}\inverse{\fgrad}$, 
by a network kinetic theory
which still suggest $K(\theta)$ to be an affine function of the temperature in a first approach \cite{Marrucci1973}.
We shall inspect generalizations of our model 
using various 
strain measures $\ctensor$ 
---like in the so-called K-BKZ fluid models \cite{bernstein-kearsley-zapas:a-study}---
in our last Section~\ref{sec:generalizations-to-other-models}.
Note that for non-conducting polymer melts, 
the Helmholtz free-energy formula 
with affine dependence on temperature is similar to phenomenological models proposed for rubbery materials 
\cite{CHADWICK1984}.
Moreover, using the inverse right Cauchy-Green measure of strain, with a time rate different than \eqref{eq:upper-convected-derivative},
shall turn out mathematically useful for a 
consistent interpretation of our 
model, see Section~\ref{sec:eulerian-lagrangian-equivalence}.

Refined physical phenomenas like the finite-extensibility of polymers 
can be introduced in \eqref{eq:fenergy} using e.g. the so-called Gent or FENE-P law instead of the Hookean law for the 
elastic energy term.
We also refer to Section~\ref{sec:generalizations-to-other-models}.
First, let us specify the simpler non-isothermal compressible Maxwell model, possibly heat-conducting,
and its mathematical properties for predicting time evolutions.

In any case, a non-linear dependence of $K$ on the temperature as suggested in e.g. \cite{1.549663} 
appears to be a non-trivial modification of our model; its analysis is 
not easily generalized from the one to be developed here.

\subsubsection{Specific internal energy}
\label{sec:specific-internal-energy}

We can derive a fundamental thermodynamic relation of compressible Maxwell 
fluids in terms of the internal energy $\ienergy = \fenergy + \temp \entropy$. 
Using 
\eqref{eq:fenergy} 
we find that
\begin{equation}
  \label{eq:0} 
  \entropy(\density, \temp, \ctensor, \efluxc)
  \defeq - \pd{\fenergy}{\temp}
  =
  -
  \pd{\fenergy_{\mathrm{s}}}{\temp}(\density, \temp) 
  -
  \frac{\alpha}{2} \left( K_1 \tr \ctensor - \kboltzmann \log \det \ctensor \right).
\end{equation}
We assume $\ppd{\fenergy_{\mathrm{s}}}{\temp} < 0$ so the relation \eqref{eq:0} can be inverted for the temperature
\begin{equation}
  \label{eq:1}
  \temp(\density, \entropy, \ctensor, \efluxc)
  =
  \pd{\ienergy_{\mathrm{s}}}{\entropy} 
  \left( 
    \density, 
    \entropy + \frac{\alpha}{2} \left( K_1 \tr \ctensor - \kboltzmann \log \det \ctensor \right) 
  \right),
\end{equation}
where $\ienergy_{\mathrm{s}}$ denotes 
the volumetric internal energy 
in 
variables $(\density, \entropy)$ such that
\begin{equation}
  \label{eq:ienergy-solvent-temp}
  \entropy_{\mathrm{s}}(\density,\temp) \defeq -\pd{\fenergy_{\mathrm{s}}}{\temp}(\density, \temp)
  \qquad
  \ienergy_{\mathrm{s}}(\density, \entropy_{\mathrm{s}}(\density, \temp))
  \defeq
  \fenergy_{\mathrm{s}}(\density, \temp) - \temp \pd{\fenergy_{\mathrm{s}}}{\temp}(\density, \temp)\,.
\end{equation}
A full thermodynamic description of compressible Maxwell 
fluids, equivalent to the fundamental thermodynamic relation in terms of the Helmholtz free energy \eqref{eq:fenergy},
is given by the full internal energy of the fluid, 
i.e.
\begin{equation}
  \label{eq:ienergy}
  \ienergy(\density, \entropy, \ctensor, \efluxc)
  =
  \ienergy_{\mathrm{s}}
  \left( 
    \density, 
    \entropy + \frac{\alpha}{2} \left( K_1 \tr \ctensor - \kboltzmann \log \det \ctensor \right) 
  \right)
  + 
  \frac{\alpha}{2} K_0 \tr \ctensor
  +
  \frac{\tau_0}{2 \kappa} \absnorm{\efluxc}^2  \,.
\end{equation}
On postulating the specific internal energy $\ienergy$, 
we assume $\ppd{\ienergy_{\mathrm{s}}}{\entropy} > 0$
to formulate an explicit formula for the specific Helmholtz free energy $\fenergy$ in turn.



\subsection{Constitutive closure relations}
\label{sec:constitutive-relations}
Our objective now 
is to derive a 
compressible heat-conducting Maxwell model by closing the generic balance laws of Section~\ref{sec:balance-laws-of-continuum-thermomechanics}
consistently with the thermodynamics 
postulated in Section \ref{sec:thermodynamic-background}. 
The derivation outlined below follows 
\cite{rajagopal.kr.srinivasa.ar:thermodynamic} see also \cite{malek.j.rajagopal.kr.ea:on}, \cite{hron.j.milos.v.ea:on}, or \cite{malek.j.prusa.v:derivation}. 
By specifying a fundamental thermodynamic relation describing the material one can determine possible constitutive relations for the Cauchy stress tensor $\cstress$, the evolution equation for the conformation tensor $\ctensor$, and the energy and entropy fluxes $\efluxc$, $\entfluxc$, in a thermodynamically consistent way. 
Indeed, a suitable choice of the aforementioned relations yields an explicit formula for the entropy production $\entprodc$ which is nonnegative, thus assuring the validity of the second law of thermodynamics. 
Of course, such a model derivation is only formal and it remains to show that motions can be well-defined,
typically as solutions to well-posed Cauchy problems: this will be the subject of Section~\ref{sec:short-time-well-posedness-for-compressible-heat-conducting-maxwell-fluid}. To proceed, let us formulate the evolution equation for the specific entropy from the postulated fundamental thermodynamic relation \eqref{eq:ienergy}.

\subsubsection{Evolution equation for entropy}
\label{sec:evolution-equation-entropy}
Given $\ienergy = \ienergy(\density, \entropy, \ctensor, \efluxc)$ supposedly
well-defined by \eqref{eq:ienergy}, the chain rule yields
\begin{equation}
  \label{eq:3}
  \density \dd{\ienergy}{t}
  =
  \density 
  \left(
    \pd{\ienergy}{\entropy} \dd{\entropy}{t}
    +
    \pd{\ienergy}{\density} \dd{\density}{t}
    +
    \tensordot{\pd{\ienergy}{\ctensor}}{\dd{\ctensor}{t}}
    +
    \vectordot{\pd{\ienergy}{\efluxc}}{\dd{\efluxc}{t}}
  \right).
\end{equation}
Using the balance of mass \eqref{eq:balance-of-mass} and the evolution equation for the internal energy \eqref{eq:ienergy-evolution-equation}, equation \eqref{eq:3} can then be reformulated as follows
\begin{equation}
  \label{eq:4}
  \tensordot{\cstress}{\gradvl}
  -
  \divergence \efluxc
  =
  \density \temp \dd{\entropy}{t}
  -
  \pressure \divergence \vecv
  +
  \density \tensordot{\pd{\ienergy}{\ctensor}}{\dd{\ctensor}{t}},
  +
  \density \vectordot{\pd{\ienergy}{\efluxc}}{\dd{\efluxc}{t}}
\end{equation}
using the temperature $\temp$ assuming $e_s$ twice differentiable with respect to its second argument (recall also \eqref{eq:1}) and the pressure $\pressure$ assuming $e_s$ differentiable with respect to its first argument, 
\begin{subequations}
  \label{eq:temperature-pressure}
  \begin{align}
    \label{eq:temperature}
    \temp(\density, \entropy, \ctensor, \efluxc)
    &\defeq
    \pd{\ienergy}{\entropy}(\density, \entropy, \ctensor, \efluxc)
    =
    \pd{\ienergy_{\mathrm{s}}}{\entropy} 
    \left( 
      \density, 
      \entropy + \frac{\alpha}{2} \left( K_1 \tr \ctensor - \kboltzmann \log \det \ctensor \right) 
    \right),
    \\
    \label{eq:pressure}
    \pressure(\density, \entropy, \ctensor, \efluxc)
    &\defeq
    \density^2 \pd{\ienergy}{\density} (\density, \entropy, \ctensor, \efluxc)
    =
    \density^2 \pd{\ienergy_{\mathrm{s}}}{\density} 
    \left( 
      \density, 
      \entropy + \frac{\alpha}{2} \left( K_1 \tr \ctensor - \kboltzmann \log \det \ctensor \right) 
    \right).
  \end{align}
\end{subequations}
Exploiting 
\eqref{eq:ienergy} and the formula \eqref{eq:upper-convected-derivative}, equation \eqref{eq:4} can then be rewritten as
\begin{multline}
  \label{eq:5}
  \tensordot{\cstress}{\gradvl}
  -
  \divergence \efluxc
  =
  \density \temp \dd{\entropy}{t}
  -
  \pressure \divergence \vecv
  \\
  +
  \frac{\alpha \density}{2}
  \tensordot{
    \left( K(\temp) \identity - \kboltzmann \temp \inverse{\ctensor} \right)
  }{
    \left( \fid{\ctensor} + \gradvl \ctensor + \ctensor \transpose{\gradvl} \right)
  }
  +
  \frac{\tau_0 \density}{\kappa} \vectordot{\efluxc}{\dd{\efluxc}{t}}
  .
\end{multline}
To cast \eqref{eq:5} into the general evolution equation for the specific entropy \eqref{eq:entropy-evolution-equation}, we divide \eqref{eq:5} by the temperature and we arrive at
\begin{multline}
  \label{eq:6}
  \density \dd{\entropy}{t}
  +
  \divergence \left( 
   \frac{\efluxc}{\temp}
  \right)
  =
  \frac{1}{\temp}
  \bigg[
    \tensordot{\left( \cstress - \alpha \density K(\temp) \ctensor \right)}{\gradsym}
    +
    \left( p + \alpha \density \kboltzmann \temp \right) \divergence \vecv
    \\
    -
    \frac{\alpha \density}{2}
    \tensordot{
      \left( K(\temp) \identity - \kboltzmann \temp \inverse{\ctensor} \right)
    }{
      \fid{\ctensor}
    }    
    - \vectordot{\efluxc}{\left(  \frac{\nabla \temp}{\temp} 
     + \frac{\tau_0 \density}{\kappa} \dd{\efluxc}{t} \right)}
  \bigg],
\end{multline}
where we used the fact that the Cauchy stress tensor $\cstress$ and the conformation tensor $\ctensor$ are symmetric. 
Finally, we split the Cauchy stress tensor $\cstress$ as
\begin{equation}
  \label{eq:cstress-decomposition}
  \cstress
  =
  \mns \identity
  +
  \traceless{\cstress}
\end{equation}
where $\mns \defeq \frac{1}{3} \tr \cstress$ is the mean normal stress 
and $\traceless{\cstress} \defeq \cstress - \frac{1}{3} (\tr \cstress) \identity$ the traceless part. 
The manipulation allows us to separate the volume-changing and volume-preserving deformations, see e.g. \cite{malek.j.prusa.v:derivation}. 
Rewriting \eqref{eq:6} as
\begin{multline}
  \label{eq:7}
  \density \dd{\entropy}{t}
  +
  \divergence \left( \frac{\efluxc}{\temp} \right)
  =
  \frac{1}{\temp}
  \bigg[
    \tensordot{\left( \traceless{\cstress} - \alpha \density K(\temp) \traceless{\ctensor} \right)}{\traceless{\gradsym}}
    \\
    +
    \left( 
      \mns + p - \frac{1}{3} \alpha \density K(\temp) \tr \ctensor + \alpha \density \kboltzmann \temp 
    \right) 
    \divergence \vecv
    \\
    -
    \frac{\alpha \density}{2}
    \tensordot{
      \left( K(\temp) \identity - \kboltzmann \temp \inverse{\ctensor} \right)
    }{
      \fid{\ctensor}
    }
    -
    \vectordot{\efluxc}{\left( \frac{\nabla \temp}{\temp}+ \frac{\tau_0 \density}{\kappa} \dd{\efluxc}{t} \right)}
  \bigg]
\end{multline}
we have finally identified four ``independent'' entropy-producing mechanisms 
on the right-hand side of \eqref{eq:7}. 
It allows one to straightforwardly propose constitutive relations in a thermodynamically consistent way,
so as to assure the nonnegativity of the individual contributions to the entropy production.

\subsubsection{Entropy production and viscoelastic constitutive relations}
\label{sec:constitutive-relations-and-entropy-production}
A possible closure for the constitutive relations that yields the desired heat-conducting Maxwell model reads as follows
\begin{subequations}
  \label{eq:constitutive-relations}
  \begin{align}
    \label{eq:constitutive-relations-mns}
    \mns 
    &=
    - 
    \pressure
    +
    \frac{1}{3} \alpha \density K(\temp) \tr \ctensor 
    -
    \alpha \density \kboltzmann \temp,
    \\
    \label{eq:constitutive-relations-traceless-cstress}
    \traceless{\cstress}
    &=
    \alpha \density K(\temp) \traceless{\ctensor},
    \\
    \label{eq:constitutive-relations-ctensor}
    \fid{\ctensor} 
    &= 
    - 
    \frac{4 K(\temp)}{\dragcoef} \ctensor + \frac{4 \kboltzmann \temp}{\dragcoef} \identity,
    \\
    \label{eq:constitutive-relations-exflux}
    \frac{\tau_0 \density}{\kappa}
    \dd{\efluxc}{t}
    &=
    - \frac{\nabla \temp}{\temp}
    -    \frac{\efluxc}{\kappa \temp},
    \\
    \label{eq:constitutive-relations-entflux}
    \entfluxc
    &=
    \frac{\efluxc}{\temp},
  \end{align}
\end{subequations}
where the 
parameter $\dragcoef>0$ can be interpreted as the hydrodynamic drag coefficient 
in molecular 
theories for polymer suspensions, 
recall Section~\ref{sec:specific-helmholtz-free-energy}.
Such a friction coefficient $\dragcoef$ a priori depends on the temperature $\temp$ 
(so $\dragcoef/\kboltzmann \temp$ actually scales like a relaxation time, characteristic of a polymer fluid) 
and possibly also on the conformation $\ctensor$. 
That dependence has important quantitative consequences for the numerical values of predicted motions,
see e.g. \cite{1.549663}. 
However, the dependence of $\dragcoef$ on $\temp$ is only \emph{algebraic} a priori (i.e. it is not through its derivatives),
so it has no influence on our (short-time) well-posedness result provided it is smooth. 
Furthermore, we shall come back to the dependence of $\dragcoef$ on $\ctensor$ later in Section~\ref{sec:generalizations-to-other-models},
to introduce finite-extensibility effects using the so-called Gent or FENE-P law.
Extensions in that direction beyond the Gent or FENE-P law are certainly desirable physically 
but are left here as issues that could possibly be addressed in the future, see our conclusion in Section~\ref{sec:conclusion}.

Entropy production $\entprodc$, i.e. the right-hand side of \eqref{eq:7}, now reads
\begin{align}
  \label{eq:entprod}
  \entprodc(\density, \temp, \ctensor, \efluxc)
  &=
  \frac{\absnorm{\efluxc}^2}{\kappa \temp^2}
  +
  \frac{2 \alpha \density}{\dragcoef \temp}
  \tensordot{
    \left( K(\temp) \identity - \kboltzmann \temp \inverse{\ctensor} \right)
  }{
    \left( K(\temp) \ctensor - \kboltzmann \temp \identity \right)
  }
  \\
  &=
   \frac{\absnorm{\efluxc}^2}{\kappa \temp^2}
  +
  \frac{2 \alpha \density}{\dragcoef \temp}
  \absnorm{K(\temp) \ctensor^{\frac{1}{2}} - \kboltzmann \temp \ctensor^{-\frac{1}{2}}}^2
  \ge
  0,
\end{align}
and the second law of thermodynamics is satisfied.
Here, $\absnorm{\tensorq{A}} \defeq \sqrt{\tensordot{\tensorq{A}}{\tensorq{A}}}$ denotes the Frobenius norm as usual.

Further, 
combining \eqref{eq:cstress-decomposition}, \eqref{eq:constitutive-relations-mns}, and \eqref{eq:constitutive-relations-traceless-cstress} we can write
\begin{subequations}
  \label{eq:maxwell-cstress}
  \begin{align}
    \label{eq:maxwell-cstress-1}
    \cstress
    &=
    -p \identity
    +
    \tensorq{S},
    \\
    \label{eq:maxwell-estress}
    \tensorq{S}
    &\defeq
    \alpha \density \left( K(\temp) \ctensor - \kboltzmann \temp \identity \right),
  \end{align}
\end{subequations}
where $\tensorq{S}$ denotes the so-called \emph{extra stress tensor}
for the compressible non-isothermal Maxwell fluids, see e.g. \cite{dressler.m.edwards.bj.ea:macroscopic}.
A straightforward computation shows that $\tensorq{S}$ satisfies a multi-dimensional version of Maxwell seminal model \cite{Maxwell01011867},
so our model 
is indeed of that \emph{viscoelastic} type.

\begin{remark} 
\label{sec:maxwell-cattaneo-law}
Fourier's law of heat conduction
\begin{equation}
  \label{eq:fourier-law}
  \efluxc
  =
  - 
  \kappa \nabla \temp
\end{equation}
is widely used 
for heat-conducting continuous media (assumed homogeneous and thermally isotropic here 
for the sake of simplicity). However, Fourier's law formally yields propagation of thermal signals with infinite speed, see e.g. \cite{joseph-preziosi:heat}, which in turn violates the principle of causality. 
To address the causality issue, the following Galilean-invariant 
Maxwell--Cattaneo law \cite{cattaneo}
\begin{equation}
  \label{eq:maxwell-cattaneo-variant-law}
  \tau_0 \density \dd{\efluxc}{t}
  =
  -  \frac{\kappa \nabla \temp}{\temp}
  -
  \frac{\efluxc}{\temp}
\end{equation}
is one standard possible alternative, which 
reduces to the standard Fourier's law when $\tau_0 \to 0^+$. 
A reason why we opt for the heat conduction law in the form of \eqref{eq:maxwell-cattaneo-variant-law}
is the quadratic dependence of the corresponding specific internal energy on the heat flux $\efluxc$,
see Section \ref{sec:specific-internal-energy}.
It allows a simple proof of the strict convexity 
in Section \ref{sec:governing-equations-symmetrizable-system-balance-laws}. 
It is also noteworthy that our model 
covers heat-insulators (non-conducting materials) in the formal limits 
$\tau_0 \to 0^+$ (so \eqref{eq:fourier-law} holds), $\kappa\to0$.
\end{remark}

Having derived the constitutive closure relations \eqref{eq:constitutive-relations} we can now formulate 
a formally-closed system of equations for the unknowns $(\density, \entropy, \efluxc, \vecv, \ctensor)$
that govern the thermomechanical evolution of the material:
\begin{subequations}
  \label{eq:governing-equations}
  \begin{align}
    \label{eq:governing-equations-density}
    \dd{\density}{t} 
    &= - \density \divergence \vecv,
    \\
    \label{eq:governing-equations-entropy}
    \density \dd{\entropy}{t}
    +
    \divergence \left( \frac{\efluxc}{\temp(\density, \entropy, \ctensor, \efluxc)}\right)
    &=
    \entprodc(\density, \entropy, \ctensor, \efluxc),   
    \\
    \label{eq:governing-equations-eflux}
    \density \dd{\efluxc}{t}
    &=
    -
    \frac{1}{\tau_0\temp(\density, \entropy, \ctensor, \efluxc)}
    \left(
      \kappa \nabla \temp(\density, \entropy, \ctensor, \efluxc)
      +
      \efluxc
    \right)
    \\
    \label{eq:governing-equations-linear-momentum}
    \density \dd{\vecv}{t} 
    &=
    \divergence \left( \cstress(\density, \entropy, \ctensor) \right) + \density \myvec{f},
    \\
    \label{eq:governing-equations-conformation-tensor}
    \fid{\ctensor} 
    &= 
    - 
    \frac{4 K(\temp(\density, \entropy, \ctensor, \efluxc))}{\dragcoef} 
    \ctensor + \frac{4 \kboltzmann \temp(\density, \entropy, \ctensor, \efluxc)}{\dragcoef} \identity,
  \end{align}
\end{subequations}
where the Cauchy stress tensor $\cstress$, the temperature $\temp$, the pressure $\pressure$ 
and the entropy production $\entprodc$ are assumed to be smooth functions of the main unknown variables 
given by \eqref{eq:maxwell-cstress}, \eqref{eq:temperature-pressure} and \eqref{eq:entprod}.
We also recall that $\dragcoef$ can be a 
smooth function of $\temp$ here at the present stage, and we refer to Section~\ref{sec:generalizations-to-other-models}
for a possible (physically-meaningful) additional dependence of $\dragcoef$ on $\ctensor$.

The governing equations~\eqref{eq:governing-equations} shall define a mechanical model
of compressible heat-conducting Maxwell fluids provided unequivocal (i.e. univocal) motions can be defined.
To that aim, we consider here the solutions to well-posed Cauchy problems 
for \eqref{eq:governing-equations} complemented by initial conditions. 

\section{Well-posedness} 
\label{sec:short-time-well-posedness-for-compressible-heat-conducting-maxwell-fluid}
Let us consider the \emph{short-time} existence of unique continuous-in-time solutions to 
Cauchy problems for \eqref{eq:governing-equations} 
like in Th.~\ref{th:short-time-well-posedness} below
as a minimal well-posedness requirement.
To use Th.~\ref{th:short-time-well-posedness}, 
the balance laws \eqref{eq:governing-equations-density}--\eqref{eq:governing-equations-linear-momentum} must be reformulated as 
a first-order system of \emph{conservation laws} \eqref{eq:system-of-balance-laws} with smooth flux and source terms. 
But the evolution equation for the conformation tensor \eqref{eq:governing-equations-conformation-tensor} cannot be cast into such a conservative form.
This makes the construction of univocal solutions to 
the associated multidimensional Cauchy problems very difficult. 
Following \cite{boyaval:viscoelastic}, we thus introduce a new tensor variable, to embed \eqref{eq:governing-equations}
in a \emph{system of conservation laws with smooth flux and relaxation source terms}.

\subsection{Governing equations 
as conservation laws with relaxation 
source terms} 
\label{sec:governing-equations}

Let us define the tensorial quantity $\tensorq{A}$ by the formula
\begin{equation}
  \label{eq:A-definition}
  \tensorq{A} \defeq \inverse{\fgrad} \ctensor \fgrad^{-\top}.
\end{equation}
The tensor $\tensorq{A}$ is assumed symmetric positive-definite, consistently with the definition of $\ctensor$.
The definition of the tensorial quantity $\tensorq{A}$ is motivated by the K-BKZ theory, 
see the seminal works \cite{kaye:non-newtonian,bernstein-kearsley-zapas:a-study} 
where $\tensorq{A}$ is used implicitly to derive an integro-differential isothermal viscoelastic fluid model similar to elastodynamics.
More specifically, it is motivated by the \emph{polyconvex} elastodynamics, 
which actually allows one to define univocal motions of hyperelastic materials after embedding 
the 
governing equations 
into 
a symmetric-hyperbolic system of conservation laws, 
see e.g. \cite{wagner:symm}.
The tensor $\tensorq{A}$ can be interpreted as a material metric
in the context of viscoelasticity, 
see \cite{boyaval:viscoelastic} for a brief description of its physical meaning. 
In any case, 
the tensor $\tensorq{A}$ satisfies
\begin{equation}
  \label{eq:A-evolution-equation}
  \dd{\tensorq{A}}{t} 
  = 
  - 
  \frac{4 K(\temp)}{\dragcoef}
  \tensorq{A}
  +
  \frac{4 \kboltzmann \temp}{\dragcoef} \inverse{\fgrad} \fgrad^{-\top}
\end{equation}
upon using the differential identity $\dd{(\inverse{\fgrad})}{t} = - \inverse{\fgrad} \dd{\fgrad}{t} \inverse{\fgrad}$
with \eqref{eq:flf} and \eqref{eq:governing-equations-conformation-tensor}.
%

Equation \eqref{eq:A-evolution-equation}
allows one to embed \eqref{eq:governing-equations-density}--\eqref{eq:governing-equations-linear-momentum} into 
a conservative 
system 
for the unknowns $\myvec{u} \defeq (\density, \density \entropy, \density \efluxc, \density \vecv, \density \fgrad, \density \tensorq{A})$:
\begin{subequations}
  \label{eq:balance-laws}
  \begin{align}
    \label{eq:balance-laws-density}
    \pd{\density}{t} + \divergence \left( \density \vecv \right) 
    &= 
    0,
    \\
    \label{eq:balance-laws-entropy}
    \pd{\left( \density \entropy \right)}{t} + \divergence \left( \density \entropy \vecv + \frac{\efluxc}{\temp} \right)
    &=
    \entprodc,  
    \\
    \label{eq:balance-laws-eflux}
    \pd{\left( \density \efluxc \right)}{t}
    +
    \divergence
    \left(
      \density \efluxc \otimes \vecv
      +
      \frac{\kappa}{\tau_0} e^{-\temp} 
      \identity
    \right)
    &=
    -
    \frac{\efluxc}{\tau_0 \temp}
    \\
    \label{eq:balance-laws-linear-momentum}
    \pd{\left( \density \vecv \right)}{t} 
    + 
    \divergence 
    \left(
      \density \vecv \otimes \vecv - \cstress
    \right) 
    &= 
    \density \myvec{f},
    \\
    \label{eq:balance-laws-fgrad}
    \pd{\left( \density \fgrad \right)}{t} 
    + 
    \divergence 
    \left( 
      \density \left( \fgrad \otimes \vecv - \vecv \otimes \transpose{\fgrad} \right) 
    \right)
    &=
    \tensorq{0},   
    \\
    \label{eq:balance-laws-tensor-A}
    \pd{\left( \density \tensorq{A} \right)}{t} 
    + 
    \divergence \left( \density \tensorq{A} \otimes \vecv \right) 
    &= 
    - 
    \frac{4 \density K(\temp)}{\dragcoef}
    \tensorq{A}
    +
    \frac{4 \density \kboltzmann \temp}{\dragcoef} \inverse{\fgrad} \fgrad^{-\top},
  \end{align}
\end{subequations}
where the entropy production $\entprodc$, the temperature $\temp$, and the Cauchy stress tensor $\cstress$ are given as functions of the unknowns $(\density, \density \entropy, \density \efluxc, \density \vecv, \density \fgrad, \density \tensorq{A})$ by the relations \eqref{eq:entprod}, \eqref{eq:temperature}, and \eqref{eq:maxwell-cstress}, respectively.
Here, the symbol $\otimes$ denotes the outer product (of tensors) and the divergence operator is always taken with respect to the last coordinate. 
This means that we have, in coordinates,
\begin{equation}
  \left[ \divergence \left( \density \tensorq{F} \otimes \vecv \right) \right]^i_\alpha
  =
  \pd{\left( \density [\tensorc{F}]^i_\alpha \vecvc^k \right)}{x^k},
\end{equation} 
and analogously in the remaining cases. 
The conservation law \eqref{eq:balance-laws-fgrad} is classically derived from 
\eqref{eq:flf} using 
\eqref{eq:balance-of-mass}, the so-called Piola identity
\begin{equation}
  \label{eq:piola-identity}
  \divergence \left( \frac{\transpose{\fgrad}}{\det \fgrad} \right)
  =
  0,
\end{equation}
assuming 
that we are dealing with a homogeneous material, so that
\begin{equation}
  \label{eq:rhor}
  \rhor \defeq \density \det \fgrad,
\end{equation}
i.e. the density of the material in the reference configuration, is constant. 


Now, we would like to use the standard results for \eqref{eq:balance-laws} which we rewrite in the form
\begin{equation}
  \label{eq:system-of-balance-laws}
  \pd{\myvec{u}}{t} + \sum_{j=1}^d \pd{\myvec{f}^j(\myvec{u})}{x^j} = \myvec{c}(\myvec{u})
\end{equation}
with smooth flux $\myvec{f}^j: \R^n \to \R^n$ and relaxation source terms $\myvec{c}: \R^n \to \R^n$
see e.g. \cite{serre_1999,dafermos.cm:hyperbolic*1,benzoni-serre:multi}.
Precisely, we consider smooth fields $\myvec{u} : \R^d \times [0, +\infty) \to \R^n$
that are equivalently solutions of the \emph{quasilinear system}
\begin{equation}
  \label{eq:quasilinear-system}
  \pd{\myvec{u}}{t} + \sum_{j=1}^d \tensorq{A}^j(\myvec{u}) \pd{\myvec{u}}{x^j} = \myvec{c}(\myvec{u}),
\end{equation}
where the $n \times n$ matrices $\tensorq{A}^j$ are given by
\begin{equation}
  \label{eq:matrices}
  \tensorq{A}^j(\myvec{u}) \defeq \pd{\myvec{f}^j(\myvec{u})}{\myvec{u}}.
\end{equation}
\begin{theorem}[Short-time well-posedness]
  \label{th:short-time-well-posedness}
  Let $\mathcal{U}$ be an open subset of $\R^n$. We assume that $\tensorq{A}^j$ and $\myvec{c}$ are $C^{\infty}$ functions of $\myvec{u} \in \mathcal{U}$ and that \eqref{eq:quasilinear-system} is \emph{Friedrichs-symmetrizable in $\mathcal{U}$} i.e. 
  there exists a $C^{\infty}$ mapping $\tensorq{S} : \mathcal{U} \to \R^{n \times n}$ such that $\tensorq{S}(\myvec{u})$ is symmetric positive-definite, and the matrices $\tensorq{S}(\myvec{u})\tensorq{A}^j(\myvec{u})$ are symmetric for all $\myvec{u} \in \mathcal{U}$.
  Let $\myvec{u}_0 \in \mathcal{U}$ and $\widetilde{\myvec{u}_0} \in H^s(\R^d; \R^n)$ with $s > 1 + \frac{d}{2}$ such that $\myvec{u}_0 + \widetilde{\myvec{u}_0}$ is compactly supported in $\mathcal{U}$. 
  
  Then, there exists $T > 0$ and a unique classical solution $\myvec{u} \in C^{1}(\R^d \times [ 0, T]; \, \mathcal{U})$ 
  to \eqref{eq:system-of-balance-laws} with the initial data $\myvec{u}(0) = \myvec{u}_0 + \widetilde{\myvec{u}_0}$. 
  
  Furthermore, $\myvec{u} - \myvec{u}_0 \in C([0, T]; H^s) \cap C^1([0, T]; H^{s-1})$.
\end{theorem}
\begin{proof}
  See Theorem 10.1 
  and Theorem 13.1 in \cite{benzoni-serre:multi}.
\end{proof}
%

To apply Theorem~\ref{th:short-time-well-posedness} to \eqref{eq:quasilinear-system}, one can show that the 
quasilinear system is Friedrichs-symmetrizable using a \emph{mathematical entropy} $s$:
\begin{theorem}[Godunov--Mock]
  \label{th:godunov-mock}
  Assume that $\mathcal{U}$ is convex and that the system of balance laws \eqref{eq:system-of-balance-laws} implies the existence of an additional balance law 
  \begin{equation}
    \label{eq:mathematical-entropy-balance}
    \pd{s(\myvec{u})}{t} + \sum_{j=1}^d \pd{F^j(\myvec{u})}{x^j} = b(\myvec{u}),
  \end{equation}
  where $s : \mathcal{U} \to \R$, $F^j : \mathcal{U} \to \R$, $j \in \{1, \dots, d\}$, and $b : \mathcal{U} \to \R$ are smooth functions and $s$ is \emph{strictly convex}. Then, the associated quasilinear system \eqref{eq:quasilinear-system} is Friedrichs-symmetrizable.
\end{theorem}
\begin{proof}
  See \cite{friedrichs-lax:systems}, or \cite{godunov:an-interesting}, \cite{godunov:elements}, and \cite{godunov:lois}.
\end{proof}

The quantity $\density E$ 
is a natural candidate for the mathematical entropy of 
\eqref{eq:balance-laws} because 
\eqref{eq:balance-of-total-energy} has the conservation form
\begin{equation}
  \label{eq:balance-law-total-energy}
  \pd{\left( \density E \right)}{t}
  +
  \divergence \left( \density E \vecv - \transpose{\cstress} \vecv + \efluxc \right)
  =
  \density \vectordot{\myvec{f}}{\vecv}
\end{equation}
with, by virtue of \eqref{eq:total-energy} and \eqref{eq:ienergy}, a specific total energy that reads
\begin{multline}
  \label{eq:total-energy-F-A}
  E
  =
  \frac{1}{2} \absnorm{\vecv}^2
  +
  \ienergy_{\mathrm{s}}
  \left( 
    \density, 
    \entropy + \frac{\alpha}{2} \left( K_1 \tr \left( \fgrad \tensorq{A} \transpose{\fgrad} \right) - \kboltzmann \log \det \left( \fgrad \tensorq{A} \transpose{\fgrad} \right) \right) 
  \right)
  \\
  +
  \frac{\tau_0}{2 \kappa} \absnorm{\efluxc}^2
  + 
  \frac{\alpha}{2} K_0 \tr \left( \fgrad \tensorq{A} \transpose{\fgrad} \right) \,.
\end{multline}

But 
$\density E$ is not strictly convex in $\myvec{u}$.
Precisely, 
$\density E$ is strictly convex with respect to $(\density, \density \entropy, \density \efluxc, \density \vecv, \density \fgrad, \density \tensorq{A})$ if and only if the function $E$ is strictly convex with respect to $(\density^{-1}, \entropy, \efluxc, \vecv, \fgrad, \tensorq{A})$, see \cite{wagner:symm} or \cite{bouchut:nonlinear}. 
However, it can be easily shown that the last term in \eqref{eq:total-energy-F-A} is \emph{not} strictly convex with respect to $(\fgrad, \tensorq{A})$. 
The situation is similar to the elastodynamics of 
hyperelastic materials see e.g. \cite{wagner:symm}:
to establish well-posedness, 
we are led to similarly use \emph{polyconvexity} and rewrite our system 
with additional independent 
quantities.

\subsection{Governing equations as a symmetrizable system} 
\label{sec:governing-equations-symmetrizable-system-balance-laws}

Let us define a symmetric positive-definite tensor $\tensorq{Y}$ via the formula
\begin{equation}
    \label{eq:Y-definition}
    \tensorq{Y} 
    \defeq 
    \tensorq{A}^{-2}.
\end{equation}
As long as $\tensorq{Y}$ and $\tensorq{A}$ remain symmetric positive-definite,
$\tensorq{Y}$ is equivalently defined by the evolution equation \eqref{eq:A-evolution-equation} for  $\tensorq{A}$
and \eqref{eq:Y-definition} as follows
\begin{equation}
    \label{eq:Y-evolution-equation}
    \dd{\tensorq{Y}}{t}
    =
    \frac{8 K(\temp)}{\dragcoef}
    \tensorq{Y}
    -
    \frac{4 \kboltzmann \temp}{\dragcoef} 
    \tensorq{Y} \left( \inverse{\fgrad} \fgrad^{-\top} \tensorq{Y}^{-\frac{1}{2}} + \tensorq{Y}^{-\frac{1}{2}} \inverse{\fgrad} \fgrad^{-\top} \right) \tensorq{Y} \defeq \tensorf{f}.
\end{equation}
Moreover, recalling the additional conservation law \eqref{eq:total-energy-F-A},
we also introduce the scalar quantity $\detY\defeq\det\inverse{\tensorq{Y}}$ as an independent variable, with evolution equation
\begin{equation}
  \label{eq:detY-evolution-equation}
  \dd{\detY}{t}
  =
    \frac{4}{\dragcoef} \detY
    \tr\left(
\kboltzmann \temp \left( \inverse{\fgrad} \fgrad^{-\top} \tensorq{Y}^{-\frac{1}{2}} + \tensorq{Y}^{-\frac{1}{2}} \inverse{\fgrad} \fgrad^{-\top} \right) \tensorq{Y}
    - 2 K(\temp) \identity \right) \defeq h 
\end{equation}
compatible with the evolution equations for $\tensorq{A}$ or $\tensorq{Y}$ 
when $\detY>0$.

The thermomechanical evolution of compressible Maxell fluids is governed by a system 
for the new set of unknowns $(\density, \density \entropy, \density \efluxc, \density \vecv, \density \fgrad, \density \tensorq{Y}, \density \detY)$:
\begin{subequations}
  \label{eq:symmetric-balance-laws}
  \begin{align}
    \label{eq:symmetric-balance-laws-density}
    \pd{\density}{t} + \divergence \left( \density \vecv \right) 
    &= 
    0,
    \\
    \label{eq:symmetric-balance-laws-entropy}
    \pd{\left( \density \entropy \right)}{t} + \divergence \left( \density \entropy \vecv + \frac{\efluxc}{\temp} \right)
    &=
    \entprodc,  
    \\
    \label{eq:symmetric-balance-laws-eflux}
    \pd{\left( \density \efluxc \right)}{t}
    +
    \divergence
    \left(
      \density \efluxc \otimes \vecv
      +
      \frac{\kappa \log\temp}{\tau_0} \identity
    \right)
    &=
    -
    \frac{\efluxc}{\tau_0 \temp}
    \\
    \label{eq:symmetric-balance-laws-linear-momentum}
    \pd{\left( \density \vecv \right)}{t} 
    + 
    \divergence 
    \left(
      \density \vecv \otimes \vecv - \cstress
    \right) 
    &= 
    \density \myvec{f},
    \\
    \label{eq:symmetric-balance-laws-fgrad}
    \pd{\left( \density \fgrad \right)}{t} 
    + 
    \divergence 
    \left( 
      \density \left( \fgrad \otimes \vecv - \vecv \otimes \transpose{\fgrad} \right) 
    \right)
    &=
    \tensorq{0},
    \\
    \label{eq:symmetric-balance-laws-tensor-Y}
    \pd{\left( \density \tensorq{Y} \right)}{t} 
    + 
    \divergence \left( \density \tensorq{Y} \otimes \vecv \right) 
    &= 
    \density \tensorf{f},
    \\
    \label{eq:symmetric-balance-laws-det-Y}
    \pd{\left( \density \detY 
    \right)}{t} 
    +
    \divergence \left( \density \detY \vecv \right) 
    &=
    \density h.
  \end{align}
\end{subequations}
where the entropy production $\entprodc$, the temperature $\temp$, the Cauchy stress tensor $\cstress$, 
and the quantities $\tensorf{f}$, $h$ are given 
by equations \eqref{eq:entprod}, \eqref{eq:temperature}, \eqref{eq:maxwell-cstress}, 
\eqref{eq:Y-evolution-equation}, \eqref{eq:detY-evolution-equation}
as functions of the unknowns $(\density, \density \entropy, \density \efluxc, \density \vecv, \density \fgrad, \density \tensorq{Y}, \density \detY)$.

It remains to identify a mathematical entropy for 
\eqref{eq:symmetric-balance-laws}. 
Denoting $\ienergyref$ a constant reference value of the internal energy,
we propose $\density \widetilde{E}$ where
\begin{equation}
  \label{eq:mathematical-entropy-definition}
  \widetilde{E}
  \defeq
  \frac{1}{2} \absnorm{\vecv}^2
  +
  \widetilde{\ienergy_{\mathrm{s}}} (\density, \entropy, \fgrad, \tensorq{Y}, \detY)
  +
  \frac{\tau_0}{2 \kappa} \absnorm{\efluxc}^2
  + 
  \frac{\alpha}{2} K_0 \tr \left( \fgrad \tensorq{Y}^{-\frac{1}{2}} \transpose{\fgrad} \right)
  +
  \frac{\ienergyref}{2}  \absnorm{\tensorq{Y}}^2,
\end{equation}
and where $\widetilde{\ienergy_{\mathrm{s}}}$ is defined---exploiting the relation \eqref{eq:rhor}---as
\begin{multline}
  \label{eq:ienergy-s-tilde}
  \widetilde{\ienergy_{\mathrm{s}}}
  (\density, \entropy, \fgrad, \tensorq{Y}, \detY)
  \defeq
  \\
  \ienergy_{\mathrm{s}}
  \left( 
    \density, 
    \entropy 
    + 
    \frac{\alpha}{2} \left( K_1 \tr \left( \fgrad \tensorq{Y}^{-\frac{1}{2}} \transpose{\fgrad} \right) 
    - 
    \kboltzmann \log \left( \left( \frac{\rhor}{\density} \right)^2 \detY^{\frac{1}{2}} \right) \right) 
  \right).
\end{multline}
Note that by virtue of \eqref{eq:total-energy-F-A} and \eqref{eq:Y-definition} it holds
\begin{equation}
  \label{eq:mathematical-entropy-reformulation}
  \widetilde{E}
  =
  E
  +
  \frac{\ienergyref}{2} \absnorm{\tensorq{Y}}^2 
\end{equation}
and thus, using the balance of the total energy \eqref{eq:balance-law-total-energy} we obtain
\begin{equation}
  \label{eq:mathematical-entropy-evolution-equation-2}
  \pd{( \density \widetilde{E} )}{t}
  +
  \divergence \left( \density \widetilde{E} \vecv - \transpose{\cstress} \vecv + \efluxc \right)
  =
  \density \vectordot{\myvec{f}}{\vecv}
  +
  \density \ienergyref 
    \tensordot{
      \tensorq{Y}
    }{  
      \tensorf{f}
    },
\end{equation}
where the right-hand side of \eqref{eq:mathematical-entropy-evolution-equation-2} is an algebraic function of the unknowns $(\density, \density \entropy, \density \efluxc, \density \vecv, \density \fgrad, \density \tensorq{Y}, \density \detY)$.

\begin{theorem}
  \label{th:strict-convexity}
  Assume that the volumetric energy functional 
  $\ienergy_{\mathrm{s}} = \ienergy_{\mathrm{s}}(\density, \entropy)$ 
  is chosen so that the function $\widetilde{\ienergy_{\mathrm{s}}}$ defined by \eqref{eq:ienergy-s-tilde} is convex with respect to $(\inverse{\density}, \entropy, \fgrad, \tensorq{Y},\detY)$ 
  and strictly convex with respect to $(\inverse{\density}, \entropy,\detY)$. 
  
  Then, the scalar quantity $\density \widetilde{E}$, where $\widetilde{E}$ is given by the formula \eqref{eq:mathematical-entropy-definition}, is strictly convex with respect to $(\density, \density \entropy, \density \efluxc, \density \vecv, \density \fgrad, \density \tensorq{Y}, \density \detY)$.
\end{theorem}

\begin{proof}
  We already know that the function $\density \widetilde{E}$ is strictly convex with respect to $(\density, \density \entropy, \density \efluxc, \density \vecv, \density \fgrad, \density \tensorq{Y}, \density \detY)$ if and only if the function $\widetilde{E}$ is strictly convex with respect to $(\density^{-1}, \entropy, \efluxc, \vecv, \fgrad, \tensorq{Z}, \tensorq{Y}, \detY)$, see \cite{wagner:symm} or \cite{bouchut:nonlinear}. 
  Next, note that the function $\tr ( \fgrad \tensorq{Y}^{-\frac{1}{2}} \transpose{\fgrad} )$ is convex with respect to $(\fgrad, \tensorq{Y})$, see \cite{lieb:convex}, and strictly convex with respect to $\fgrad$. As a consequence, the sum
  \begin{equation}
    \label{eq:p1}
    \frac{\alpha}{2} K_0 \tr \left( \fgrad \tensorq{Y}^{-\frac{1}{2}} \transpose{\fgrad} \right)
    +
    \frac{\ienergyref}{2}        \absnorm{\tensorq{Y}}^2
  \end{equation}
  is strictly convex with respect to $(\fgrad, \tensorq{Y})$. 
  Exploiting the assumption about the function $\widetilde{\ienergy_{\mathrm{s}}}$ we then obtain that the sum
  \begin{equation}
    \label{eq:p2}
    \widetilde{\ienergy_{\mathrm{s}}} (\density, \entropy, \fgrad, \tensorq{Y}, \detY)
    + 
    \frac{\alpha}{2} K_0 \tr \left( \fgrad \tensorq{Y}^{-\frac{1}{2}} \transpose{\fgrad} \right)
    +
    \frac{ \ienergyref}{2} 
      \absnorm{\tensorq{Y}}^2
  \end{equation}
  is strictly convex with respect to $(\inverse{\density}, \entropy, \fgrad, \tensorq{Y}, \detY)$. 
  Further, the terms $\frac{1}{2} \absnorm{\vecv}^2$ and $\frac{\tau_0}{2 \kappa} \absnorm{\efluxc}^2$ are strictly convex with respect to $\vecv$ and $\efluxc$, respectively. 
  Trivially, these terms as well as the expression \eqref{eq:p2} are also convex functions with respect to $(\density^{-1}, \entropy, \efluxc, \vecv, \fgrad, \tensorq{Y}, \detY)$. 
  Consequently, $\widetilde{E}$ is strictly convex with respect to $(\density^{-1}, \entropy, \efluxc, \vecv, \fgrad, \tensorq{Y}, \detY)$.
\end{proof}

We are now ready to formulate our first result.
\begin{theorem}
  \label{th:pg}
  Let the 
  volumetric term of the energy be described by the polytropic gas equation of state
  \begin{equation}
    \label{eq:pg-ienergy}
    \ienergy_{\mathrm{s}}(\density, \entropy)
    =
    \cheatvolsol \tempref \left( \frac{\density}{\densityref} \right)^{\gamma - 1} 
    \exp \left( \frac{\entropy}{\cheatvolsol} \right),
  \end{equation}
  where $\cheatvolsol$ denotes a specific heat capacity at constant volume, 
  $\gamma > 1$ is a constant referred to as the adiabatic exponent,
  and $\tempref,\densityref$ denote a constant reference temperature and a constant reference density, respectively,
  which can be chosen at will as they serve for normalization purposes only.
  
  Then, the scalar quantity $\density \widetilde{E}$, where $\widetilde{E}$ is given by the formula \eqref{eq:mathematical-entropy-definition}, is a mathematical entropy of the system of balance laws \eqref{eq:symmetric-balance-laws}.
\end{theorem}

To prove Theorem \ref{th:pg}, it is useful to recall an elementary fact from convex optimization 
see e.g. \cite[Theorem 2.3.10]{peressini-sullivan-uhl:the}.
\begin{lemma}
  \label{lem:1}
  Let $X \subset \R^n$, $n \in \N$ be a convex set and let $g : X \to \R$ be convex (resp. strictly convex) and $f : \R \to \R$ be non-decreasing (resp. increasing) and convex. Then, $f \circ g$ is convex (resp. strictly convex).
\end{lemma}

\begin{proof}[of Theorem \ref{th:pg}]
  We already know that the quantity $\density \widetilde{E}$ satisfies an additional balance law, see \eqref{eq:mathematical-entropy-evolution-equation-2}. By Theorem \ref{th:strict-convexity}, to prove strict convexity of $\density \widetilde{E}$, it suffices to show that the function $\widetilde{\ienergy_{\mathrm{s}}}$ defined by \eqref{eq:ienergy-s-tilde} satisfies the assumptions of the theorem. The fundamental relation \eqref{eq:pg-ienergy} yields an explicit formula for 
  $\widetilde{\ienergy_{\mathrm{s}}}$,
  \begin{multline}
    \label{eq:pg-ienergy-tilde}
    \widetilde{\ienergy_{\mathrm{s}}}
    (\density, \entropy, \fgrad, \tensorq{Y}, \detY)
    =
    \\
    \cheatvolsol \tempref \left( \frac{\density}{\densityref} \right)^{\gamma - 1} 
    \exp 
    \left(
      \frac{\entropy}{\cheatvolsol} + \frac{\alpha K_1}{2 \cheatvolref} \tr \left( \fgrad \tensorq{Y}^{-\frac{1}{2}} \transpose{\fgrad} \right)
    \right)
    \left( \frac{\rhor}{\density} \right)^{-\frac{\alpha \kboltzmann}{\cheatvolsol}} \detY^{-\frac{\alpha \kboltzmann}{4 \cheatvolsol}}.
  \end{multline}
  
  First, we show that $\widetilde{\ienergy_{\mathrm{s}}}$ is strictly convex with respect to $(\inverse{\density}, \entropy, \detY)$. Let us define an auxiliary function
  \begin{equation}
    \label{eq:pg-f}
    f(x, y, z)
    \defeq
    \frac{\exponential{q y}}{x^p z^r},
  \end{equation}
  where $(x, y, z) \in D_{f} \defeq \{ (x, y, z) \in \left( 0, +\infty \right) \times \R \times \left( 0, +\infty \right) \}$ and $p, q, r > 0$. To prove that the function $\widetilde{\ienergy_{\mathrm{s}}}$ is strictly convex with respect to $(\density^{-1}, \entropy, \detY)$ for arbitrary values of the material parameters $\alpha$, $\cheatvolsol$, $\gamma$, $\tempref$, $\densityref$ it suffices to prove that the function $f$ is strictly convex with respect to $(x, y, z)$ for arbitrary values of the parameters $p, q, r$. Let $H_i$ denote the $i$-th leading principal minor of the Hessian matrix $\myvec{H}$ corresponding to the function $f$. An elementary calculation yields
  \begin{equation}
    \label{eq:pg-determinants}
    H_1 = \frac{p(p+1) \exponential{q y}}{x^{p+2} z^r},
    \qquad
    H_2 = \frac{q^2 p \exponential{2 q y}}{x^{2p+2} z^{2r}},
    \qquad
    H_3 = \frac{q^2 p r \exponential{3 q y}}{x^{3p+2} z^{3r + 2}},
  \end{equation}
  and we immediately see that, by Sylvester's criterion, the Hessian matrix $\myvec{H}$ is positive definite on $D_f$. Consequently, the function $f$ is strictly convex on $D_{f}$.
  
  Second, we prove that $\widetilde{\ienergy_{\mathrm{s}}}$ is convex with respect to $(\inverse{\density}, \entropy, \fgrad, \tensorq{Y}, \detY)$. Note that the formula \eqref{eq:pg-ienergy-tilde} can be conveniently written as
  \begin{multline}
      \label{eq:pg-ienergy-tilde-2}
    \widetilde{\ienergy_{\mathrm{s}}}
    (\density, \entropy, \fgrad, \tensorq{Y}, \detY)
    =
    \\
    C
    \exp
    \left(
      -
      p \log \left( \inverse{\density} \right)
      +
      q \entropy
      -
      r \log \left( \detY \right)
      +
      s \tr \left( \fgrad \tensorq{Y}^{-\frac{1}{2}} \transpose{\fgrad} \right)
    \right),      
  \end{multline}
  where $C, p, q, r, s > 0$ are positive constants whose explicit formulae can be found easily. Since the scalar function $\exp$ is non-decreasing and convex, let us only concentrate on the inner function
  \begin{equation}
    g(\density, \entropy, \fgrad, \tensorq{Y}, \detY)
    \defeq
    -
    p \log \left( \inverse{\density} \right)
    +
    q \entropy
    -
    r \log \left( \detY \right)
    +
    s \tr \left( \fgrad \tensorq{Y}^{-\frac{1}{2}} \transpose{\fgrad} \right). 
  \end{equation}
  As we already know since \cite{lieb:convex}, the function $\tr ( \fgrad \tensorq{Y}^{-\frac{1}{2}} \transpose{\fgrad} )$ is convex with respect to $(\fgrad, \tensorq{Y})$, recall the proof of Th. \ref{th:strict-convexity},
  so $g$ is clearly convex with respect to $(\inverse{\density}, \entropy, \fgrad, \tensorq{Y}, \detY)$. Employing Lemma \ref{lem:1} we then conclude that $\widetilde{\ienergy_{\mathrm{s}}}$ is convex with respect to the unknowns $(\inverse{\density}, \entropy, \fgrad, \tensorq{Y}, \detY)$.
\end{proof}

Finally, having proved Theorem \ref{th:pg},
it is now possible to formulate a short-time well-posedness result of the Cauchy problem associated with the system 
\eqref{eq:symmetric-balance-laws} for $\myvec{u}= (\density, \density \entropy, \density \efluxc, \density \vecv, \density \fgrad, \density \tensorq{Y}, \density \detY)$, when the  volumetric specific internal energy 
is given by \emph{the polytropic gas equation of state} \eqref{eq:pg-ienergy}.
Recall Theorems \ref{th:short-time-well-posedness}\&\ref{th:godunov-mock}, 
and the conservation form \eqref{eq:system-of-balance-laws} of \eqref{eq:symmetric-balance-laws} 
using \emph{smooth} flux $\myvec{f}^j(\myvec{u})$, $j=1\ldots d$ and source $\myvec{c}(\myvec{u})$ 
well-defined for all $\myvec{u}\in \mathcal{V}$ in the convex open set of $\R^{24}$
\begin{equation}
  \label{eq:V-set}
  \mathcal{V} \defeq (0, +\infty) \times \R \times \R^3 \times \R^3 \times \R^{9} \times \R_{\mathrm{sym},>}^{6} \times (0, +\infty) \cap \{\temp(\myvec{u})>0\}
\end{equation}
where $\R_{\mathrm{sym},>}^{6}$ denotes the convex subset of $\R^6$ whose elements can be 
rearranged by convention to form $3 \times 3$ symmetric positive definite matrices.


\begin{theorem}
  \label{th:maxwell-pg-short-time-well-posedness}
  Let $\mathcal{U}$ be an open subset of the convex set $\mathcal{V}$ given by \eqref{eq:V-set}. 
  Consider the system of balance laws \eqref{eq:symmetric-balance-laws} governing the motion of compressible heat-conducting Maxwell viscoelastic fluid with a volumetric specific internal energy 
  given by the polytropic gas equation of state \eqref{eq:pg-ienergy}. 
  Assume that $\myvec{u}_0 \in \mathcal{U}$ and $\widetilde{\myvec{u}_0} \in H^s(\R^3; \BBB{\R^{30}}\AAA{\R^{24}})$ with $s > \frac{5}{2}$ such that $\myvec{u}_0 + \widetilde{\myvec{u}_0}$ is compactly supported in $\mathcal{U}$. Then, there exists $T > 0$ and a unique classical solution $\myvec{u} \in C^{1}(\R^3 \times [ 0, T]; \, \mathcal{U})$ of the Cauchy problem associated with \eqref{eq:symmetric-balance-laws} and the initial data $\myvec{u}(0) = \myvec{u}_0 + \widetilde{\myvec{u}_0}$. Furthermore, $\myvec{u} - \myvec{u}_0 \in C([0, T]; H^s) \cap C^1([0, T]; H^{s-1})$.
\end{theorem}

At this point, we should stress that the components of the solution $\myvec{u} = (\density, \density \entropy, \density \efluxc, \density \vecv, \density \fgrad, \density \tensorq{Y}, \density \detY)$ act as independent quantities. 
However, we are interested by the case where initial conditions for the Cauchy problem
$$
\myvec{u}_0 + \widetilde{\myvec{u}_0}
= (\density_0, \density_0 {\entropy}_0, \density_0 {\efluxc}_0, \density_0 \vecv_0, \density_0 \fgrad_0, \density_0 \tensorq{Y}_0, \density_0 \detY_0)
$$
satisfy $\detY_0 = 1/\det \tensorq{Y}_0$, and $\rhor = \density_0 \det \tensorq{F}_0$
with $\tensorq{F}_0 = \pd{\deformation_t|_{t=0}}{\myvec{a}}$
for some initial deformation $\deformation_t|_{t=0} : \R^3 \to \R^3$ such that 
$\vecv_0 \circ \deformation_t|_{t=0}= \partial_t\deformation_t|_{t=0}$,
consistently with physical interpretation.
Moreover, we would like to preserve the interpretation for $t > 0$, that is
$\detY(\myvec{x}, t) = 1/\det \tensorq{Y}(\myvec{x}, t)$, and 
$\rhor = \density(\myvec{x}, t) \det \tensorq{F}(\myvec{x}, t)$
with $\tensorq{F} = \pd{\deformation_t}{\myvec{a}}$
for some deformation $\deformation_t : \R^3 \to \R^3$ such that 
$\vecv \circ \deformation_t = \partial_t\deformation_t$ holds almost everywhere,
consistently with physical interpretation.
Now, whereas $\detY(\myvec{x}, t) = 1/\det \tensorq{Y}(\myvec{x}, t)$ is obviously preserved for $t > 0$ by the unique smooth solutions
to \eqref{eq:Y-evolution-equation}, \eqref{eq:detY-evolution-equation}, 
it is neither the case of $\tensorq{F} = \pd{\deformation_t}{\myvec{a}}$,
$\vecv_t \circ \deformation_t = \partial_t\deformation_t$ 
(i.e. the existence of a Lagrangian description equivalent almost everywhere to our Eulerian description)
nor of $\rhor = \density(\myvec{x}, t) \det \tensorq{F}(\myvec{x}, t)$ by consequence.
Indeed, our system of \emph{conservation} laws
has not preserved \eqref{eq:flf}, see e.g. \cite{wagner-1994}. 
%
This motivates 
a variation of our model in Section \ref{sec:generalizations-to-other-models}
which allows to preserve equivalence between Eulerian and Lagrangian descriptions with a \emph{conservation} law,
moreover 
with physical grounds in applications to polymer \emph{melts}. 


Lastly, to conclude this section, we would like to stress that our well-posedness result
holds for short times where $\temp>0$ is preserved, with $\temp$ interpreted as the temperature.
Note that it preserves $K(\theta)>0$ in turn, with $K$ interpreted as the stiffness.
Using $\entropy = - \pd{\fenergy}{\temp}$, 
the evolution equation for the specific entropy \eqref{eq:governing-equations-entropy} rewrites 
\begin{equation}
  \label{eq:temp-equation-1}
  \density
  \left(
    -
    \ppd{\fenergy}{\temp}
    \dd{\temp}{t}
    -
    \frac{\partial^2 \fenergy}{\partial \temp \partial \density}
    \dd{\density}{t}
    -
    \vectordot{\frac{\partial^2 \fenergy}{\partial \temp \partial \efluxc}}{\dd{\efluxc}{t}}
    -
    \tensordot{
      \frac{\partial^2 \fenergy}{\partial \temp \partial \ctensor}
    }{
      \dd{\ctensor}{t}
    }
  \right)
  +
  \divergence \left( \frac{\efluxc}{\temp}\right)
  =
  \entprodc.
\end{equation}
Then, multiplying equation \eqref{eq:temp-equation-1} by the temperature and exploiting the conservation of mass, see \eqref{eq:balance-of-mass}, we can further write
\begin{equation}
  \label{eq:temp-equation-2}
  \density
  \cheatvolsol
  \dd{\temp}{t}
  +
  \temp
  \pd{\pressure}{\temp}
  \divergence \vecv
  -
  \density \temp
  \vectordot{\frac{\partial^2 \fenergy}{\partial \temp \partial \efluxc}}{\dd{\efluxc}{t}}
  -
  \density
  \temp
  \tensordot{
    \frac{\partial^2 \fenergy}{\partial \temp \partial \ctensor}
  }{
    \dd{\ctensor}{t}
  }
  +
  \temp
  \divergence \left( \frac{\efluxc}{\temp}\right)
  =
  \temp \entprodc,
\end{equation}
where we recall 
$  \pressure(\density, \temp, \ctensor, \efluxc)
  \defeq
  \density^2
  \pd{\fenergy}{\density}
  (\density, \temp, \ctensor, \efluxc)
  =
  \density^2
  \pd{\fenergy_{\mathrm{s}}}{\density}
  (\density, \temp)
$ using \eqref{eq:pressure} and the standard thermodynamic relations,
and where we have introduced 
the specific heat capacity at constant volume of the \emph{solvent} $\cheatvolsol$ 
\begin{equation}
  \label{eq:cheatvolsol}
  \cheatvolsol(\density, \temp)
  \defeq
  -
  \temp
  \ppd{\fenergy_{\mathrm{s}}}{\temp}
  (\density, \temp)
  =
  -
  \temp
  \ppd{\fenergy}{\temp}
  (\density, \temp, \ctensor, \efluxc)  
.
\end{equation}
Finally, using the definition of the specific Helmholtz free energy \eqref{eq:fenergy}, 
the formula  \eqref{eq:entprod} for the entropy production,
the evolution equations for the energy flux \eqref{eq:governing-equations-eflux} 
and for the conformation tensor \eqref{eq:governing-equations-conformation-tensor}, 
we arrive at 
\begin{multline}
  \label{eq:temperature-evolution-equation}
  \density \cheatvolsol \dd{\temp}{t}
  +
  \divergence \efluxc
  =
  -
  \temp
  \left(
    \pd{p}{\temp}(\density, \temp, \ctensor, \efluxc)
    +
    \alpha \density \kboltzmann
  \right)
  \divergence \vecv
  +
  \vectordot{\efluxc}{\nabla \log \temp}
  \\
  +
  \frac{\absnorm{\efluxc}^2}{\kappa}
  +
  \frac{2 \alpha \density}{\dragcoef} K(\temp) 
  \tensordot{\identity}{\left( K(\temp) \ctensor - \kboltzmann \temp \identity \right)}.
\end{multline}
Clearly, the evolution equation \eqref{eq:temperature-evolution-equation} preserves $\temp>0$ for the smooth 
\emph{short} time
evolutions 
in Theorem \ref{th:maxwell-pg-short-time-well-posedness}.
Note in passing that due to 
the first two terms on the right-hand side, \eqref{eq:temperature-evolution-equation} cannot be formulated as a conservation law.

\section{Some variations of the Maxwell model}
\label{sec:generalizations-to-other-models}

We consider here three variations of the Maxwell model 
in Theorem \ref{th:maxwell-pg-short-time-well-posedness}.

First, we consider conservation laws that extend the elastodynamics to viscoelastic flows
using a different measure of strain than $\ctensor$ used in the previous section for compressible Maxwell fluids.
This is motivated mathematically by \cite{wagner-1994,wagner:symm} 
where it is shown that only such a variation allows a correct interpretation of the (small-time) smooth solutions.
This is also motivated physically by the success of such strain measures in some applications to polymer fluids; melts in particular.
That is why we formulate the first variation by \emph{adding} an energy term to the Helmholtz free-energy 
which is a function of that new strain measure,
like in the useful class of K-BKZ fluids which inspires our formulation. 
To obtain a well-posedness result similar to Theorem \ref{th:maxwell-pg-short-time-well-posedness} with that variation,
it suffices to establish a mathematical entropy, which is done in Section \ref{sec:eulerian-lagrangian-equivalence}.
Noteworthily, in addition to a well-posedness result similar to Theorem \ref{th:maxwell-pg-short-time-well-posedness},
that variation can be shown \emph{fully consistent with its physical interpretation} as time evolves
(smooth solutions preserve their initial interpretation).

Second, we consider an independent variation which uses a different elastic energy term motivated by physics:
real strain is bounded above. 
We show that our well-posedness result still holds with a standard elastic energy term
that accounts for polymers \emph{finite-extensibility}.

Last, we consider in Section \ref{sec:nasg} another independent variation regarding the volumetric energy term,
since the polytropic gas equation of state is 
not very well-suited for polymeric 
\emph{liquids}. A small-time well-posedness result similar to Theorem \ref{th:maxwell-pg-short-time-well-posedness}
can be established provided a mathematical entropy is established, and $\rho\in (0,b)$ is ensured on choosing an initial condition
$\rho_0\in (0,b)$.

Many other 
variations are certainly possible. Note our three ones here have mathematical and/or physical motivation.
Different interesting variations to be considered in the future are mentionned in our conclusive Section \ref{sec:conclusion}.

\subsection{Adding another 
measure of strain 
like in 
K-BKZ fluids} 
\label{sec:eulerian-lagrangian-equivalence}

Let us postulate the specific Helmholtz free energy in the form
\begin{multline}
  \label{eq:fenergy-ele}
  \fenergy(\density, \temp, \ctensor_1, \ctensor_2, \efluxc)
  =
  \fenergy_{\mathrm{s}}(\density, \temp)
  +
  \frac{\alpha}{2}
  \left(
    K^{(1)}(\temp) \tr \ctensor_1
    -
    \kboltzmann \temp \log \det \ctensor_1
  \right)
  \\
  +
  \frac{\alpha}{2}
  \left(
    K^{(2)}(\temp) \tr \ctensor_2
    -
    \kboltzmann \temp \log \det \ctensor_2
  \right)
  +
  \frac{\tau_0}{2 \kappa} \absnorm{\efluxc}^2
  ,
\end{multline}
where $\ctensor_1$, $\ctensor_2$ are two symmetric positive-definite 
tensors measuring strain and $K^{(1)}$, $K^{(2)}$ are 
two corresponding stiffness, affine functions of the temperature
\begin{subequations}
  \label{eq:elastic-spring-factors-affine}
  \begin{align}
    K^{(1)}
    &\defeq
    K^{(1)}_0 + K^{(1)}_1 \temp,
    \\
    K^{(2)}
    &\defeq
    K^{(2)}_0 + K^{(2)}_1 \temp, 
  \end{align}
\end{subequations}
where $K^{(1)}_0, K^{(1)}_1, K^{(2)}_0, K^{(2)}_1 > 0$ are constant parameters. 
Following Section \ref{sec:thermodynamic-background}, a fundamental thermodynamic relation in terms of the internal energy 
\begin{multline}
  \label{eq:el-ienergy}
  \ienergy(\density, \entropy, \ctensor_1, \ctensor_2, \efluxc)
  =
  \\
  \ienergy_{\mathrm{s}} 
  \left( 
    \density, 
    \entropy
    +
    \frac{\alpha}{2}
    \left(
      K^{(1)}_1 \tr \ctensor_1
      +
      K^{(2)}_1 \tr \ctensor_2
      -
      \kboltzmann \log \det \left( \ctensor_1 \ctensor_2 \right)
    \right)
  \right) 
  \\
  + 
  \frac{\alpha}{2} K^{(1)}_0 \tr \ctensor_1
  + 
  \frac{\alpha}{2} K^{(2)}_0 \tr \ctensor_2
  +
  \frac{\tau_0}{2 \kappa} \absnorm{\efluxc}^2
\end{multline}
is equivalent upon assuming differentiability of $\fenergy_{\mathrm{s}},\ienergy_{\mathrm{s}}$.
Further, we assume that $\ctensor_1$ is a measure of strain similar to $\ctensor$ in the previous sections
(i.e. similar to the left Cauchy-Green deformation tensor $\fgrad\transpose{\fgrad}$)
while we assume $\ctensor_2/\rho^2$ similar to the inverse right Cauchy-Green deformation tensor $\transposei{\fgrad}\inverse{\fgrad}$. 
Precisely, we postulate 
time rates are given by the following 
objective time derivatives 
\begin{subequations}
  \label{eq:od}
  \begin{align}
    \fid{\ctensor_1}
    &\defeq
    \dd{\ctensor_1}{t}
    -
    \gradvl \ctensor_1
    -
    \ctensor_1 \transpose{\gradvl},
    \\   
    \label{eq:lower-convected-derivative-alternative}
    \lfid{\ctensor_2}
    &\defeq
    \dd{\ctensor_2}{t}
    +
    \ctensor_2 \gradvl
    +
    \transpose{\gradvl} \ctensor_2
    -
    2 \left( \tr \gradvl \right) \ctensor_2.
  \end{align}
\end{subequations}
Next, we can derive constitutive relations following 
Section \ref{sec:constitutive-relations}. 
We first take the material derivative of the relation $\ienergy = \ienergy(\density, \entropy, \ctensor_1, \ctensor_2, \efluxc)$
\begin{equation}
  \density \dd{\ienergy}{t}
  =
  \density
  \left(
    \pd{\ienergy}{\density}
    \dd{\density}{t}
    +
    \pd{\ienergy}{\entropy}
    \dd{\entropy}{t}
    +
    \tensordot{\pd{\ienergy}{\ctensor_1}}{\dd{\ctensor_1}{t}}
    +
    \tensordot{\pd{\ienergy}{\ctensor_2}}{\dd{\ctensor_2}{t}}
    +
    \vectordot{\pd{\ienergy}{\efluxc}}{\dd{\efluxc}{t}}
  \right).
\end{equation}
The mass balance \eqref{eq:balance-of-mass} and the 
equation \eqref{eq:ienergy-evolution-equation} for internal energy 
then yield
\begin{multline}
  \label{eq:ddd}
  \density \dd{\entropy}{t}
  +
  \divergence \left( \frac{\efluxc}{\temp} \right)
  =
  \frac{1}{\temp}
  \bigg(
    \tensordot{\cstress}{\gradsym}
    +
    \pressure \divergence \vecv
    \\
    -
    \tensordot{\pd{\ienergy}{\ctensor_1}}{\dd{\ctensor_1}{t}}
    -
    \tensordot{\pd{\ienergy}{\ctensor_2}}{\dd{\ctensor_2}{t}}
    -
    \vectordot{\efluxc}{\left( \frac{\nabla \temp}{\temp}+ \frac{\tau_0 \density}{\kappa} \dd{\efluxc}{t} \right)}
  \bigg),
\end{multline}
where we introduced the temperature $\temp$ and the pressure $\pressure$ defined by the standard thermodynamic relations
\eqref{eq:temperature-pressure}. 
Exploiting the explicit relation for the internal energy \eqref{eq:el-ienergy} and 
\eqref{eq:od}, equation \eqref{eq:ddd} rewrites
\begin{multline}
  \label{eq:11}
  \density \dd{\entropy}{t}
  +
  \divergence \left( \frac{\efluxc}{\temp} \right)
  =
  \frac{1}{\temp}
  \bigg[
    \tensordot{
      \left( \cstress - \alpha \density \left( K^{(1)}(\temp) \ctensor_1 - K^{(2)}(\temp) \ctensor_2 \right) \right)
    }{
      \gradsym
    }
    \\
    +
    \left( \pressure - \alpha \density \left( K^{(2)}(\temp) \tr \ctensor_2 - 3 \kboltzmann \temp \right) \right) \divergence \vecv
    -
    \frac{\vectordot{\efluxc}{\nabla \temp}}{\temp}
    \\
    -
    \frac{\alpha \density}{2}
    \tensordot{
      \left( K^{(1)}(\temp) \identity - \kboltzmann \temp \inverse{\ctensor_1} \right)
    }{
      \fid{\ctensor_1}
    }
    -
    \frac{\alpha \density}{2}
    \tensordot{
      \left( K^{(2)}(\temp) \identity - \kboltzmann \temp \inverse{\ctensor_2} \right)
    }{
      \lfid{\ctensor_2}
    }
  \bigg].
\end{multline}
This finally leads to the following set of constitutive relations 
\begin{subequations}
  \label{eq:el-constitutive-relations}
  \begin{align}
  \cstress
  &=
  -
  \pressure \identity
  +
  \tensorq{S},
  \\
  \tensorq{S}
  &=
  \alpha \density
  \left[
    K^{(1)}(\temp) \ctensor_1 - K^{(2)}(\temp) \ctensor_2
    -
    \left( K^{(2)}(\temp) \tr \ctensor_2 - 3 \kboltzmann \temp \right) \identity
  \right],
  \\
  \label{eq:el-constitutive-relations-ctensor-1}
  \fid{\ctensor_1}
  &=
  -
  \frac{4 K^{(1)}(\temp)}{\dragcoef}
  \ctensor_1
  +
  \frac{4 \kboltzmann \temp}{\dragcoef} \identity,
  \\
  \label{eq:el-constitutive-relations-ctensor-2}
  \lfid{\ctensor_2}
  &=
  -
  \frac{4 K^{(2)}(\temp)}{\dragcoef}
  \ctensor_2
  +
  \frac{4 \kboltzmann \temp}{\dragcoef} \identity,
  \end{align}
\end{subequations}
which guarantee that the right-hand side of \eqref{eq:11}, i.e. the entropy production, is non-negative. 
From \eqref{eq:el-constitutive-relations} one also recognizes the viscoelastic character of the Hookean material considered here,
insofar as $\tensorq{S}$ follows Maxwell law in the infinitesimal strain limit \cite{boyaval:viscoelastic}.

Next, with a view to applying Theorem \ref{th:pg}, we proceed 
as in Section \ref{sec:governing-equations-symmetrizable-system-balance-laws} and introduce 
symmetric positive-definite tensors $\tensorq{A}_1$, $\tensorq{A}_2$ via the formulae
\begin{subequations}
  \begin{align}
    \tensorq{A}_1
    &\defeq
    \inverse{\fgrad} \ctensor_1 \fgrad^{-\top},
    \\
    \tensorq{A}_2
    &\defeq
    \inverse{\left( \Cof \fgrad \right)} \ctensor_2 \left( \Cof \fgrad \right)^{-\top}.
  \end{align}
\end{subequations}
Using 
\eqref{eq:el-constitutive-relations-ctensor-1}, \eqref{eq:el-constitutive-relations-ctensor-2} for 
$\ctensor_1$, $\ctensor_2$, along with the kinematic relations
\begin{subequations}
  \label{eq:kinematic-relations}
  \begin{align}
    \label{eq:kinematic-relations-fgrad}
    \dd{\fgrad}{t}
    &=
    \gradvl \fgrad,
    \\
    \label{eq:kinematic-relations-cofactor}
    \dd{(\Cof \fgrad)}{t}
    &=
    \left(
      \left( \tr \gradvl \right) \identity
      -
      \transpose{\gradvl}
    \right)
    \Cof \fgrad,
  \end{align}
\end{subequations}
one can easily identify evolution equations
\begin{subequations}
  \begin{align}
    \dd{\tensorq{A}_1}{t} 
    &= 
    - 
    \frac{4 K^{(1)}(\temp)}{\dragcoef}
    \tensorq{A}_1
    +
    \frac{4 \kboltzmann \temp}{\dragcoef} \inverse{\fgrad} \fgrad^{-\top},
    \\
    \dd{\tensorq{A}_2}{t} 
    &= 
    - 
    \frac{4 K^{(2)}(\temp)}{\dragcoef}
    \tensorq{A}_2
    +
    \frac{4 \kboltzmann \temp}{\dragcoef} \inverse{\left( \Cof \fgrad \right)} \left( \Cof \fgrad \right)^{-\top},
  \end{align}  
\end{subequations}
so we can formulate a system of \emph{conservation} 
laws with relaxation source terms 
in the 
unknowns $(\density, \density \entropy, \density \efluxc, \density \vecv, \density \fgrad, \density \transpose{\Cof \fgrad}, \density \tensorq{A}_1, \density \tensorq{A}_2)$ which governs 
our new compressible Maxwell fluids of a more general type, with 
two 
strain measures.
We recall that 
$\density \transpose{\Cof \fgrad}\equiv\inverse{\fgrad}$ 
is governed by the conservation law (see \cite{wagner:symm})
\begin{equation}
  \label{eq:cofactor-conservation-law}
  \pd{(\inverse{\fgrad})}{t}
  +
  \nabla\otimes
  \left(
    \inverse{\fgrad}\vecv
  \right)
  =
  0,
\end{equation}
where we have used the notation $\left[\nabla\otimes\left(\inverse{\fgrad}\vecv\right)\right]^\alpha_i = \pd{}{x^i}\left( [\fgradc^{-1}]^\alpha_j \vecvc^j \right)$.
%
To establish (short-time) well-posedness, 
it remains to identify a \emph{symmetric-hyperbolic} system of conservation laws, typically endowed with a mathematical entropy.

Introducing symmetric positive-definite tensors $\tensorq{Y}_1\defeq\tensorq{A}_1^{-2}$, $\tensorq{Y}_2\defeq\tensorq{A}_2^{-2}$
and scalars $\detY_1\defeq\det \inverse{\tensorq{Y}_1}$, $\detY_2\defeq\det \inverse{\tensorq{Y}_2}$ 
as in Section \ref{sec:governing-equations-symmetrizable-system-balance-laws}, we obtain a system of conservation 
laws with relaxation source terms for the set of unknowns 
\begin{equation}
 \label{eq:set}
(\density, \density \entropy, \density \efluxc, \density \vecv, \density \fgrad, \density \transpose{\Cof \fgrad}, \density \tensorq{Y}_1, \density \tensorq{Y}_2, \density  \detY_1, \density \detY_2).
\end{equation}
That system is endowed with  a
mathematical entropy 
\begin{multline}
  \label{eq:el-mathematical-entropy-definition}
  \widetilde{E}
  \defeq
  \frac{1}{2} \absnorm{\vecv}^2
  +
  \widetilde{\ienergy_{\mathrm{s}}} (\density, \entropy, \fgrad, \transpose{\Cof \fgrad}, \tensorq{Y}_1, \tensorq{Y}_2, \detY_1, \detY_2)
  +
  \frac{\tau_0}{2 \kappa} \absnorm{\efluxc}^2
  \\
  + 
  \frac{\alpha}{2} K^{(1)}_0 \tr \left( \tensorq{Y}_1^{-\frac{1}{2}} \transpose{\fgrad} \fgrad \right)
  + 
  \frac{\alpha}{2} K^{(2)}_0 \tr \left( \tensorq{Y}_2^{-\frac{1}{2}} \transpose{\Cof \fgrad} \Cof \fgrad \right)
  \\
  +
  \frac{1}{2} \ienergyref 
  \left( 
    \absnorm{\tensorq{Y}_1}^2
    +
    \absnorm{\tensorq{Y}_2}^2 
  \right),
\end{multline}
where 
---exploiting 
\eqref{eq:rhor} and 
$\det (\Cof \fgrad) = (\det \fgrad)^2$---
$\widetilde{\ienergy_{\mathrm{s}}}$ is defined as
\begin{multline}
  \label{eq:el-ienergy-s-tilde}
  \widetilde{\ienergy_{\mathrm{s}}}
  (\density, \entropy, \fgrad, \transpose{\Cof \fgrad}, \tensorq{Y}_1, \tensorq{Y}_2, \detY_1, \detY_2)
  \defeq
  \\
  \ienergy_{\mathrm{s}} 
  \Bigg(
    \density, 
    \entropy 
    +
    \frac{\alpha}{2}
    \Bigg[
      K^{(1)}_1 \tr \left( \tensorq{Y}_1^{-\frac{1}{2}} \transpose{\fgrad} \fgrad \right)
      +
      K^{(2)}_1 \tr \left( \tensorq{Y}_2^{-\frac{1}{2}} \transpose{\Cof \fgrad} \Cof \fgrad \right)  
      \\
      -
      \kboltzmann 
      \log 
      \left( 
        \left( \frac{\rhor}{\density} \right)^6 \detY_1^{\frac{1}{2}} \detY_2^{\frac{1}{2}}
      \right) 
    \Bigg]
  \Bigg)
\end{multline}
if 
$\widetilde{e_{\mathrm{s}}}$ is convex with respect to 
$(\density, \entropy, \fgrad, \transpose{\Cof \fgrad}, \tensorq{Y}_1, \tensorq{Y}_2, \detY_1, \detY_2)$ and strictly convex with respect to $(\inverse{\density}, \entropy, \detY_1, \detY_2)$. 
%
The procedure for verifying that $\widetilde{\ienergy_{\mathrm{s}}}$ satisfies the convexity properties for specific instances of the solvent contribution is completely analogous to the one in Section \ref{sec:governing-equations-symmetrizable-system-balance-laws} and we refrain from giving a detailed derivation. 
Let us only prove that the polytropic gas equation of state yields the quantity $\widetilde{\ienergy_{\mathrm{s}}}$ being strictly convex with respect to $(\inverse{\density}, \entropy, \detY_1, \detY_2)$. 


From the fundamental thermodynamic relation for the polytropic gas in terms of the internal energy, see \eqref{eq:pg-ienergy}, we obtain
\begin{multline}
    \label{eq:pg-ienergy-tilde-ele}
    \widetilde{\ienergy_{\mathrm{s}}}
    (\density, \entropy, \fgrad, \transpose{\Cof \fgrad}, \tensorq{Y}_1, \tensorq{Y}_2, \detY_1, \detY_2)
    =
    \cheatvolsol \tempref \left( \frac{\density}{\densityref} \right)^{\gamma - 1} 
    \\
    \times
    \exp 
    \left(
      \frac{\entropy}{\cheatvolsol}
      +
      \frac{\alpha K^{(1)}_1}{2 \cheatvolref} \tr \left( \fgrad \tensorq{Y}_1^{-\frac{1}{2}} \transpose{\fgrad} \right)
      +
      \frac{\alpha K^{(2)}_1}{2 \cheatvolref} \tr \left( \left( \Cof \fgrad \right) \tensorq{Y}_2^{-\frac{1}{2}} \transpose{\left( \Cof \fgrad \right)} \right)
    \right)
    \\
    \times
    \left(
      \left( \frac{\rhor}{\density} \right)^6 \detY_1^{\frac{1}{2}} \detY_2^{\frac{1}{2}}
    \right)^{-\frac{\alpha \kboltzmann}{2 \cheatvolsol}}.
\end{multline}
  We are interested in proving the strict convexity of $\widetilde{\ienergy_{\mathrm{s}}}$ with respect to the unknowns $(\inverse{\density}, \entropy, \detY_1, \detY_2)$. For this purpose, let us define an auxiliary function
  \begin{equation}
    \label{eq:pg-f-ele}
    f(x, y, z_1, z_2)
    \defeq
    \frac{\exponential{q y}}{x^p z_1^r z_2^r},
  \end{equation}
  where $(x, y, z_1, z_2) \in D_{f} \defeq \{ (x, y, z_1, z_2) \in \left( 0, +\infty \right) \times \R \times \left( 0, +\infty \right) \times \left( 0, +\infty \right)\}$ and $p, q, r > 0$. To prove that the function $\widetilde{\ienergy_{\mathrm{s}}}$ is strictly convex with respect to $(\density^{-1}, \entropy, \detY_1, \detY_2)$ for arbitrary values of the material (or reference) parameters $\alpha$, $\cheatvolsol$, $\gamma$, $\tempref$, $\densityref$ it suffices to prove that the function $f$ is strictly convex with respect to $(x, y, z_1, z_2)$ for arbitrary values of the parameters $p, q, r$. Let $H_i$ denote the $i$-th leading principal minor of the Hessian matrix $\myvec{H}$ corresponding to the function $f$. An elementary calculation yields
  \begin{align}
    \label{eq:pg-determinants-ele}
    H_1 &= \frac{p(p+1) \exponential{q y}}{x^{p+2} z_1^r z_2^r},
    &
    H_2 &= \frac{q^2 p \exponential{2 q y}}{x^{2p+2} z_1^{2r} z_2^{2r}},
    \\
    H_3 &= \frac{q^2 p r \exponential{3 q y}}{x^{3p+2} z_1^{3r + 2} z_2^{3r}},
    &
    H_4 &= \frac{q^2 p r^2 \exponential{4 q y}}{x^{4p+2} z_1^{4r + 2} z_2^{4r+2}},
  \end{align}
  and we 
  see that, by Sylvester's criterion, the Hessian matrix $\myvec{H}$ is positive definite on $D_f$. Consequently, the function $f$ is strictly convex on $D_{f}$.

We can finally use Theorem \ref{th:short-time-well-posedness} 
to establish well-posedness of the system of conservation laws for \eqref{eq:set},
which is symmetric-hyperbolic by application of Theorem \ref{th:godunov-mock}, 
as in Theorem \ref{th:pg} for our previous (simpler) Maxwell fluid.
Moreover, in 
addition to Theorem \ref{th:pg}, we can establish the full consistency for all existence times $t>0$
of the smooth solutions with their physical interpretation set initially.
Indeed, equation \eqref{eq:cofactor-conservation-law} obviously allows smooth solutions to preserve the Piola's identity
\begin{equation}
  \label{eq:piola-identity-2}
  \rot \left( \inverse{\fgrad} \right) = 0,
\end{equation}
thus 
the 
interpretation of $\inverse{\fgrad}$ as a gradient $\nabla \inverse{\deformation_t}$, see \cite{wagner-1994}.
In turn, it allows the preservation of all other desired interpretations that were missing after our Theorem \ref{th:pg} in the previous Section.

A last comment is in order here, before we shift to another variation of our initial 
Maxwell model.
We have chosen here a Helmholtz free energy \eqref{eq:fenergy-ele} that is \emph{one particular} sum of three 
energy terms, functions of three particular strain measures: 
\begin{multline*}
  \fenergy 
  =
  \fenergy_{\mathrm{s}}(\density, \temp)
  +
  \frac{\alpha}{2}
  \Bigg(
    K^{(1)}(\temp) \tr \left( \tensorq{Y}_1^{-\frac{1}{2}} \transpose{\fgrad} \fgrad \right) 
  +
    K^{(2)}(\temp) \tr \left( \tensorq{Y}_2^{-\frac{1}{2}} \transpose{\Cof \fgrad} \Cof \fgrad \right) 
    \\
    -
    \kboltzmann 
    \log 
    \left( 
      \left( \frac{\rhor}{\density} \right)^6 \detY_1^{\frac{1}{2}} \detY_2^{\frac{1}{2}}
    \right) 
  \Bigg)
  +
  \frac{\tau_0}{2 \kappa} \absnorm{\efluxc}^2
\,.
\end{multline*}
We have showed that this choice provides a \emph{fully consistent} system of conservation laws 
that naturally extends standard elastodynamics to 
fading memory materials 
with a \emph{mathematical entropy}---thereby ensuring (short time) well-posedness.
Certainly, this is only one choice 
among many possible other Helmholtz free energies using different functions of different strain measures.
We show below two physical variations of the Hookean energy term and of the volumetric energy term, respectively,
keeping the same strain measures.

But one could also use different strain measures,
for instance assuming $\ctensor_2$ (and not $\ctensor_2/\rho^2$) 
similar to the inverse right Cauchy-Green deformation tensor $\transposei{\fgrad}\inverse{\fgrad}$ (sometimes termed Finger deformation),
then with another 
time rate than \eqref{eq:od} (the objective lower-convected 
derivative, in effect). 
And we are not aware of a clear rationale that would lead one to choose a particular formula function of various strain measures.
  
\subsection{Introducting finite-extensibility effects}
\label{sec:compressible-heat-conducting-fene-p-fluid}

We propose a variation of the Hookean elastic energy term in the compressible Maxwell models above
that takes into account the finite-extensibility of polymers. 
For the sake of simplicity, we consider only one ``distortional'' strain measure (beyond the purely volumetric variable $\density$),
chosen similar to the left Cauchy-Green deformation tensor as in Section \ref{sec:short-time-well-posedness-for-compressible-heat-conducting-maxwell-fluid}.
We consider the fundamental thermodynamic relation in terms of the specific Helmholtz free energy with an elastic term of FENE-P type \cite{peterlin.a:hydrodynamics}
\begin{multline}
  \label{eq:fene-p-fenergy}
  \fenergy(\density, \temp, \ctensor, \efluxc)
  =
  \fenergy_{\mathrm{s}}(\density, \temp)
  +
  \frac{\tau_0}{2 \kappa} \absnorm{\efluxc}^2
  \\
  +
  \frac{\alpha}{2}
  \left(
    -
    K(\temp) b^2 \log \left( 1 - \frac{\tr \ctensor}{b^2} \right)
    -
    \kboltzmann \temp \log \det \ctensor
  \right),
\end{multline}
where the additional constant parameter $b$ denotes the maximum allowable extension of the polymer chains. 
We again assume the 
stiffness $K$ affine in the temperature, see \eqref{eq:elastic-spring-factor-affine-function}.
But contrary to \cite{dressler.m.edwards.bj.ea:macroscopic}, we opt for the simple setting in which the material parameter $b$ does not depend on the temperature. 

The new (distortional) elastic energy term, in comparison to the specific Helmholtz free energy for the Maxwell model \eqref{eq:fenergy}, 
yields a different constitutive relations for the Cauchy stress tensor $\cstress$, 
and a different evolution equation for the conformation tensor $\ctensor$ (a different source term, in effect)
as well as a different formula for the entropy production $\entprodc$. 
However, these changes can be derived easily and
they are not essential to the well-posedness 
of the model.

Regarding well-posedness, the only substantial difference appears in the formulation of 
the mathematical entropy, see \eqref{eq:mathematical-entropy-definition} and \eqref{eq:ienergy-s-tilde}. There, the terms
\begin{equation}
  \label{eq:101}
  \tr \left( \fgrad \tensorq{Y}^{-\frac{1}{2}} \transpose{\fgrad} \right)
\end{equation}
are replaced by
\begin{equation}
  \label{eq:102}
  -b^2
  \log \left( 1 - \frac{\tr \left( \fgrad \tensorq{Y}^{-\frac{1}{2}} \transpose{\fgrad} \right)}{b^2} \right).  
\end{equation}
Since the rest of the formula for $\widetilde{E}$ stays the same, it remains to verify that the convexity properties of the term \eqref{eq:101} are preserved for \eqref{eq:102} as well. Specifically, \eqref{eq:102} must be convex with respect to $(\fgrad, \tensorq{Y})$ and strictly convex with respect to $\fgrad$. However, this is a simple consequence of Lemma \ref{lem:1} since the scalar function $- \log (1 - x/b^2)$ is increasing and convex and the aforementioned convexity properties hold for the term $\tr (\fgrad \tensorq{Y}^{-\frac{1}{2}} \transpose{\fgrad})$.

We thus conclude that an analogous short-time well-posedness result as for the Maxwell model, see Section \ref{sec:short-time-well-posedness-for-compressible-heat-conducting-maxwell-fluid}, can be obtained for the FENE-P model.

\subsection{Noble--Abel stiffened-gas equation of state}
\label{sec:nasg}

We propose a variation of the volumetric term in our Maxwell model,
the Noble--Abel stiffened-gas equation of state 
\cite{le-metayer.o.saurel.r:noble-abel},
that is capable of describing gaseous as well as liquid phases by a suitable choice of parameters. 

The fundamental thermodynamic relation for the Noble--Abel stiffened gas in terms of the specific internal energy (a complete equation of state) reads 
\begin{equation}
  \label{eq:nasg-ienergy}
  \ienergy_{\mathrm{s}}(\density, \entropy)
  =
  \cheatvolsol \tempref
  \left(
    \frac{\density}{\densityref}
    \frac{1}{1 - b \density}
  \right)^{\gamma - 1}
  \exponential{\frac{\entropy}{\cheatvolsol}}
  +
  \left( \frac{1}{\density} - b \right) \pressure_{\infty}
  +
  q,
\end{equation}
where $\cheatvolsol$ denotes the specific heat capacity at constant volume of the solvent, $\gamma > 1$ is a constant referred to as the adiabatic exponent, and $b$, $q$, and $\pressure_{\infty}$ denote constant material parameters specific for the given fluid. The symbols $\tempref$ and $\densityref$ denote constant reference temperature and constant reference density, respectively.
%

An analogous (short-time) well-posedness result as Theorem \ref{th:pg} in Section \ref{sec:governing-equations-symmetrizable-system-balance-laws} 
can be achieved. It suffices to prove:
\begin{theorem}
  If the volumetric energy term 
  is 
  the Noble--Abel stiffened-gas equation of state \eqref{eq:nasg-ienergy}, then
  the scalar quantity $\density \widetilde{E}$, where $\widetilde{E}$ is given by 
  \eqref{eq:mathematical-entropy-definition}, is a mathematical entropy of the system of balance laws \eqref{eq:symmetric-balance-laws}.
\end{theorem}

\begin{proof}
  Similarly as in the proof of Theorem \ref{th:pg} it suffices to show that the function $\widetilde{\ienergy_{\mathrm{s}}}$ defined by \eqref{eq:ienergy-s-tilde} satisfies the assumptions of Theorem \ref{th:strict-convexity}, i.e. that $\widetilde{\ienergy_{\mathrm{s}}}$ is strictly convex with respect to $(\inverse{\density}, \entropy, \detY)$ and convex with respect to $(\inverse{\density}, \entropy, \fgrad, \tensorq{Y}, \detY)$. The fundamental relation \eqref{eq:nasg-ienergy} yields an explicit formula for the quantity $\widetilde{\ienergy_{\mathrm{s}}}$,
  \begin{multline}
    \label{eq:nasg-ienergy-tilde}
    \widetilde{\ienergy_{\mathrm{s}}}
    (\density, \entropy, \fgrad, \tensorq{Y}, \detY)
    =
    \\
    \cheatvolsol \tempref 
    \left(
      \frac{\density}{\densityref}
      \frac{1}{1 - b \density}
    \right)^{\gamma - 1}
    \exp 
    \left(
      \frac{\entropy}{\cheatvolsol} + \frac{\alpha K_1}{2 \cheatvolref} \tr \left( \fgrad \tensorq{Y}^{-\frac{1}{2}} \transpose{\fgrad} \right)
    \right)
    \\
    \times
    \left( \frac{\rhor}{\density} \detY \right)^{-\frac{\alpha \kboltzmann}{\cheatvolsol}}
    +
    \left( \frac{1}{\density} - b \right) \pressure_{\infty}
    +
    q.
  \end{multline}  
  Note that the last two terms on the right-hand side of \eqref{eq:nasg-ienergy-tilde} are convex with respect to $(\inverse{\density}, \entropy, \fgrad, \tensorq{Y}, \detY)$. Hence, it suffices to prove (strict) convexity of the remaining term.
  
  First, we show that $\widetilde{\ienergy_{\mathrm{s}}}$ is strictly convex with respect to $(\inverse{\density}, \entropy, \detY)$. Let us define an auxiliary function
  \begin{equation}
    \label{eq:nasg-f}
    f(x, y, z)
    \defeq
    \frac{\exponential{q y}}{\left( x - b \right)^p x^r z^r},
  \end{equation}
  where $(x, y, z) \in D_{f} \defeq \{ (x, y, z) \in \left( b, +\infty \right) \times \R \times \left( 0, +\infty \right) \}$ and $p, q, r > 0$. To prove that the first term on the right-hand side of \eqref{eq:nasg-ienergy-tilde} is strictly convex with respect to $(\inverse{\density}, \entropy, \detY)$ for arbitrary values of the material parameters $\alpha$, $\cheatvolsol$, $\gamma$, $b$, $\tempref$, $\densityref$ it suffices to prove that the function $f$ is strictly convex with respect to $(x, y, z)$ for arbitrary values of the parameters $p, q, r$. Let $H_i$ denote the $i$-th leading principal minor of the Hessian matrix $\myvec{H}$ corresponding to the function $f$. An elementary yet tedious calculation yields
  \begin{subequations}
    \label{eq:nasg-determinants}
    \begin{align}
      \label{eq:nasg-h1}
      H_1 
      &= 
      \frac{\exponential{q y}}{(x-b)^{p+2} x^{r+2} z^r}
      \left[
        (p+r)(1+p+r)x^2 - 2br(1+p+r) + b^2 r(1+r)
      \right],
      \\
      \label{eq:nasg-h2}
      H_2 
      &= 
      \frac{q^2 \exponential{2 q y}}{(x-b)^{2p+2} x^{2r+2} z^{2r}}
      \left[
        (p+r)x^2 - 2 b r x + b^2 r
      \right],
      \\
      \label{eq:nasg-h3}
      H_3 
      &= 
      \frac{q^2 r \exponential{3 q y}}{(x - b)^{3p+2} x^{3r + 2} z^{3r + 2}}
      \left[
        (p+r)x^2 - 2 b r x + b^2 r
      \right].
    \end{align}
  \end{subequations}
  By a direct computation one can verify that the discriminants of the quadratic functions appearing in \eqref{eq:nasg-h1} and \eqref{eq:nasg-h2} read $-4b^2pr(1+p+r)$ and $-4b^2pr$, respectively. Hence, by Sylvester's criterion, the Hessian matrix $\myvec{H}$ is positive definite on $D_f$. Consequently, the function $f$ is strictly convex on $D_{f}$.
    
  Second, we prove that $\widetilde{\ienergy_{\mathrm{s}}}$ is convex with respect to $(\inverse{\density}, \entropy, \fgrad, \tensorq{Y}, \detY)$. Note that the first term on the right-hand side of \eqref{eq:nasg-ienergy-tilde} can be conveniently written 
  \begin{equation}
    \label{eq:nasg-ienergy-tilde-2}
    C
    \exp
    \left(
      -
      p \log \left( \inverse{\density} - b \right)
      -
      r \log \left( \inverse{\density} \right)
      +
      q \entropy
      -
      r \log \left( \detY \right)
      +
      s \tr \left( \fgrad \tensorq{Y}^{-\frac{1}{2}} \transpose{\fgrad} \right)
    \right),      
  \end{equation}
  where $C, p, q, r, s > 0$ are positive constants whose explicit formulae can be found easily. Similarly as in the proof of Theorem \ref{th:pg}, by showing that the argument of the exponential function in \eqref{eq:nasg-ienergy-tilde-2} is convex and by exploiting Lemma \ref{lem:1} we finally conclude that $\widetilde{\ienergy_{\mathrm{s}}}$ is convex with respect to $(\inverse{\density}, \entropy, \fgrad, \tensorq{Y}, \detY)$.
\end{proof}

\section{Conclusion}
\label{sec:conclusion}

We have pursued here our extension of the elastodynamics system of conservation laws 
for compressible viscoelastic fluids with fading memory using relaxing structural tensors 
\cite{boyaval:viscoelastic}
to \emph{non-isothermal} flows.

We have also considered generalized (Hookean--elastic) Maxwell fluids using different measures of strain.
This is physically justified (for application to polymer \emph{melts} and other rubbery materials beyond polymer solutions e.g.)
and mathematically interesting insofar as it ensures smooth solutions that are fully consistent with their interpretation
(through a Lagrangian description equivalent to our Eulerian model).

Our model asymptotically contains 
numerous known models 
in the isothermal incompressible limits.
It offers a sound mathematical framework to answer physical issues raised by usual models
e.g. \cite{mackay.at.phillips.tn:on,bollada.pc.phillips.tn:on},
insofar as univocal (smooth) solutions can be defined (on short times).
%

One question to be addressed in the future is the class of Helmholtz free-energy formulae, functions of various strain measures,
that can be considered. 
For many real materials, adding 
terms (volumetric or Hookean--elastic, possibly with finite-extensibility effects, but functions
of only \emph{one} strain measure among various possible, plus of \emph{independent} structural parameters with different time scales)
might suffice to cover useful applications. 


A second---related---question is how to improve the description of thermal effects,
especially on large temperature ranges throughout ``phase transitions''.
On large temperature ranges, one has to cover various time scales,
so a Helmholtz free-energy sum of many terms,
using many structural tensors with many different relaxation time scales, is certainly one option to be considered. 
To correctly match thermal effects then, letting those relaxation times $\dragcoef$ depend on $\temp$ might suffice,
e.g. to model glass-forming polymers.
Precise mathematical formulations seem desirable for sharp numerical predictions.